\documentclass[final]{siamltex}
\usepackage[leqno]{amsmath}
\usepackage{amsbsy}
\usepackage{amssymb}
\usepackage[dvips]{graphicx}

\title{Uncertainty Quantification by Alternative Decompositions of Multivariate Functions \thanks{This work was supported by the U.S. National Science Foundation under Grants Numbers CMMI-0969044 and CMMI-1130147.}}

\author{Sharif Rahman\thanks{Applied Mathematical \& Computational Sciences, The University of Iowa, Iowa City, IA 52242 ({\tt sharif-rahman@uiowa.edu}).}}

\begin{document}

\maketitle

\begin{abstract}
This article advocates factorized and hybrid dimensional decompositions
(FDD/HDD), as alternatives to analysis-of-variance dimensional decomposition (ADD), for second-moment statistical analysis of multivariate functions. New formulae revealing the relationships between component functions of FDD and ADD are proposed. While ADD or FDD is relevant
when a function is strongly additive or strongly multiplicative, HDD,
whether formed linearly or nonlinearly, requires no specific dimensional
hierarchies. Furthermore, FDD and HDD lead to alternative definitions
of effective dimension, reported in the current literature only for
ADD. New closed-form or analytical expressions are derived for univariate
truncations of all three decompositions, followed by mean-squared error analysis of univariate ADD, FDD, and HDD approximations. The analysis finds appropriate conditions when one approximation is better than the other. Numerical results affirm the theoretical finding that HDD is ideally suited to a general function approximation that may otherwise require higher-variate
ADD or FDD truncations for rendering acceptable accuracy in stochastic solutions.
\end{abstract}



\pagestyle{myheadings}
\thispagestyle{plain}
\markboth{S. RAHMAN}{DIMENSIONAL DECOMPOSITIONS}

\section{Introduction}

Uncertainty quantification of complex systems entails stochastic computing
for a large number of random variables. Although the sampling-based
methods can solve any stochastic problem, they generally require numerous
deterministic trials and are, therefore, cost-prohibitive when each
analysis demands expensive finite-element or similar numerical calculations.
Existing analytical or approximate methods require additional assumptions,
mostly for computational expediency, that begin to deteriorate when
the input-output mapping is highly nonlinear and the input variance
is arbitrarily large. Furthermore, truly high-dimensional problems
are all but impossible to solve using most existing methods, including
numerical integration. The root deterrence to practical computability
is often related to the high dimension of the multivariate integration
or interpolation problem, known as the \emph{curse of dimensionality}
\cite{bellman57}. The dimensional decomposition of a multivariate
function \cite{hoeffding48,sobol69,rahman12,kuo10} addresses the
curse of dimensionality to some extent by developing an input-output
behavior of complex systems with low \emph{effective dimension} \cite{caflisch97},
wherein the degrees of interactions between input variables attenuate
rapidly or vanish altogether.

A prominent variant of dimensional decomposition is the well-known
analysis-of-variance or ANOVA dimensional decomposition (ADD), first
presented by Hoeffding in the 1940s in relation to his seminal work on $U$-statistics
\cite{hoeffding48}. Since then, ADD has been studied by numerous
researchers in disparate fields of mathematics \cite{sobol03,hickernell95}, statistics \cite{owen03,efron81}, finance \cite{griebel10}, and basic and applied sciences \cite{rabitz99}, including engineering disciplines, mostly for uncertainty quantification \cite{xu04,rahman08,rahman09}. However, ADD constitutes a finite
sum of lower-dimensional component functions of a multivariate function,
and is, therefore, predicated on the additive nature of a function
decomposition. In contrast, when a response function is dominantly
of a multiplicative nature, suitable multiplicative-type decompositions,
such as factorized dimensional decomposition (FDD) \cite{tunga05},
should be explored. But existing truncations of FDD are limited to
only univariate or bivariate approximations, because FDD component
functions of three or more variables have yet to be determined.
No error analyses exist comparing ADD and FDD, even for respective
univariate approximations. Nonetheless, ADD or FDD is relevant as
long as the dimensional hierarchy of a stochastic response is also
additive or multiplicative. Unfortunately, the dimensional structure
of a response function, in general, is not known \emph{a priori}.
Therefore, indiscriminately using ADD or FDD for general stochastic
analysis is not desirable. Further complications may arise when a
complex system exhibits a response that is dominantly neither additive
nor multiplicative. In the latter case, hybrid approaches coupling
both additive and multiplicative decompositions, preferably selected
optimally, are needed. For such decompositions, it is unknown which
truncation parameter should be selected when compared with that for
ADD or FDD. Is it possible to solve a stochastic problem by selecting
a lower truncation parameter for hybrid decompositions than for ADD
or FDD? If the answer is yes, then a significant, positive impact
on high-dimensional uncertainty quantification is anticipated. These
enhancements, some of which are indispensable, should be pursued without
sustaining significant additional cost.

The purpose of this paper is threefold. Firstly, a brief exposition
of ADD and FDD is given in Section 3. A theorem, proven herein, reveals
the relationship between all component functions of FDD and ADD, so
far available only for univariate and bivariate component functions.
Three function classes, comprising purely additive functions, purely
multiplicative functions, and their mixtures, are examined to illustrate
when and how one decomposition or approximation is better than the
other. Secondly, a new hybrid approach optimally blending ADD and
FDD approximations, referred to as hybrid dimensional decomposition
(HDD), is presented in Section 4 for second-moment analysis.
Both linear and nonlinear mixtures of ADD and FDD approximations are
supported. Gaining insights from FDD and HDD, alternative definitions of effective
dimension are proposed. Thirdly, Section 5 reports new explicit formulae
for respective univariate approximations derived from ADD, FDD, and
HDD. The mean-squared error analyses pertaining to univariate
ADD, FDD, and HDD approximations are also described. Numerical results from
four elementary yet illuminating examples and a practical engineering problem are reported in Sections
3 through 5 as relevant. There are nine new theoretical results stated
or proved in this paper: Theorems \ref{thm:1}, \ref{thm:2}, and \ref{thm:3}, Corollaries
\ref{cor:1}, \ref{cor:3}, and \ref{cor:4}, Propositions
\ref{prop:3} and \ref{prop:4}, and Lemma \ref{lem:1}.  Mathematical notations and conclusions
are defined or drawn in Sections 2 and 6, respectively.

%\end{document}

\section{Notations}

Let $\mathbb{N}$, $\mathbb{N}_{0}$, $\mathbb{R}$, and $\mathbb{R}_{0}^{+}$
represent the sets of positive integer (natural), non-negative integer,
real, and non-negative real numbers, respectively. For $k\in\mathbb{N}$,
denote by $\mathbb{R}^{k}$ the $k$-dimensional Euclidean space and
by $\mathbb{R}^{k\times k}$ the set of $k\times k$ real-valued matrices.
These standard notations will be used throughout the paper.

Let $(\Omega,\mathcal{F},P)$ be a complete probability space, where
$\Omega$ is a sample space, $\mathcal{F}$ is a $\sigma$-field on
$\Omega$, and $P:\mathcal{F}\to[0,1]$ is a probability measure.
With $\mathcal{B}^{N}$ representing the Borel $\sigma$-field on
$\mathbb{R}^{N}$, $N\in\mathbb{N}$, consider an $\mathbb{R}^{N}$-valued
random vector $\mathbf{X}:=(X_{1},\cdots,X_{N}):(\Omega,\mathcal{F})\to(\mathbb{R}^{N},\mathcal{B}^{N})$,
which describes the statistical uncertainties in all system and input
parameters of a high-dimensional stochastic problem. The probability
law of $\mathbf{X}$ is completely defined by its joint probability
density function $f_{\mathbf{X}}:\mathbb{R}^{N}\to\mathbb{R}_{0}^{+}$.
Assuming independent coordinates of $\mathbf{X}$, its joint probability
density $f_{\mathbf{X}}(\mathbf{x})=\Pi_{i=1}^{i=N}f_{i}(x_{i})$
is expressed by a product of marginal probability density functions
$f_{i}$ of $X_{i}$, $i=1,\cdots,N$, defined on the probability
triple $(\Omega_{i},\mathcal{F}_{i},P_{i})$ with a bounded or an
unbounded support on $\mathbb{R}$. For a given $u\subseteq\{1,\cdots,N\}$,
$f_{-u}(\mathbf{x}_{-u}):=\prod_{i=1,i\notin u}^{N}f_{i}(x_{i})$
defines the marginal density function of $\mathbf{X}_{-u}:=\mathbf{X}_{\{1,\cdots,N\}\backslash u}$.

\section{ANOVA and Factorized Dimensional Decompositions}

Let $y(\mathbf{X}):=y(X_{1},\cdots,X_{N}$), a real-valued, measurable
transformation on $(\Omega,\mathcal{F})$, define a high-dimensional
stochastic response of interest and $\mathcal{L}_{2}(\Omega,\mathcal{F},P)$
represent a Hilbert space of square-integrable functions $y$ with
respect to the induced generic measure $f_{\mathbf{X}}(\mathbf{x})d\mathbf{x}$
supported on $\mathbb{R}^{N}$. Two useful dimensional decompositions, namely, ADD and FDD of $y$, are briefly described as follows.

\subsection{ADD}

The ANOVA dimensional decomposition of $y$, expressed by the recursive
form \cite{rahman12,kuo10,sobol03},
\begin{subequations}
\begin{align}
y(\mathbf{X}) & ={\displaystyle \sum_{u\subseteq\{1,\cdots,N\}}y_{u}(\mathbf{X}_{u})},\label{1a}\\
y_{\emptyset} & =\int_{\mathbb{R}^{N}}y(\mathbf{x})f_{\mathbf{X}}(\mathbf{x})d\mathbf{x},\label{1b}\\
y_{u}(\mathbf{X}_{u}) & ={\displaystyle \int_{\mathbb{R}^{N-|u|}}y(\mathbf{X}_{u},\mathbf{x}_{-u})}f_{-u}(\mathbf{x}_{-u})d\mathbf{x}_{-u}-{\displaystyle \sum_{v\subset u}}y_{v}(\mathbf{X}_{v}),\label{1c}
\end{align}
\label{1}
\end{subequations}
\!\!is a finite, hierarchical expansion
in terms of its input variables with increasing dimensions, where
$u\subseteq\{1,\cdots,N\}$ is a subset with the complementary set
$-u=\{1,\cdots,N\}\backslash u$ and cardinality $0\le|u|\le N$,
and $y_{u}$ is a $|u|$-variate component function describing a constant
or the interactive effect of $\mathbf{X}_{u}=(X_{i_{1}},\cdots,X_{i_{|u|}})$,
$1\leq i_{1}<\cdots<i_{|u|}\leq N$, a subvector of $\mathbf{X}$,
on $y$ when $|u|=0$ or $|u|>0$. The summation in (\ref{1a}) comprises
$2^{N}$ terms, with each term depending on a group of variables indexed
by a particular subset of $\{1,\cdots,N\}$, including the empty set
$\emptyset$. In (\ref{1c}), $(\mathbf{X}_{u},\mathbf{x}_{-u})$
denotes an $N$-dimensional vector whose $i$th component is $X_{i}$
if $i\in u$ and $x_{i}$ if $i\notin u.$ When $u=\emptyset$, the
sum in (\ref{1c}) vanishes, resulting in the expression of the constant
function $y_{\emptyset}$ in (\ref{1b}). When $u=\{1,\cdots,N\}$,
the integration in (\ref{1c}) is on the empty set, reproducing (\ref{1a})
and hence finding the last function $y_{\{1,\cdots,N\}}$. Indeed,
all component functions of $y$ can be obtained by interpreting literally
(\ref{1c}). The non-constant component functions satisfy
the annihilating conditions \cite{rahman12,kuo10,sobol03},
\[
\int_{\mathbb{R}}y_{u}(\mathbf{x}_{u})f_{i}(x_{i})dx_{i}=0\;\text{for}\; i\in u,
\]
resulting in two remarkable properties described by Propositions \ref{prop:1}
and \ref{prop:2}.

\begin{proposition}
\label{prop:1}The ADD component functions $y_{u}$, $\emptyset\ne u\subseteq\{1,\cdots,N\}$,
have zero means, i.e.,
\[
\mathbb{E}\left[y_{u}(\mathbf{X}_{u})\right]=\int_{\mathbb{R}^{|u|}}y_{u}(\mathbf{x}_{u})f_{u}(\mathbf{x}_{u})d\mathbf{x}_{u}=0.
\]

\end{proposition}

\begin{proposition}
\label{prop:2}Two distinct ADD component functions $y_{u}$ and $y_{v}$,
where $\emptyset\ne u\subseteq\{1,\cdots,N\}$, $\emptyset\ne v\subseteq\{1,\cdots,N\}$,
and $u\neq v$, are orthogonal, i.e., they satisfy the property,
\[
\mathbb{E}\left[y_{u}(\mathbf{X}_{u})y_{v}(\mathbf{X}_{v})\right]=\int_{\mathbb{R}^{|u\cup v|}}y_{u}(\mathbf{x}_{u})y_{v}(\mathbf{x}_{v})f_{u\cup v}(\mathbf{x}_{u\cup v})d\mathbf{x}_{u\cup v}=0.
\]

\end{proposition}

%\begin{remark}
{\em Remark 1.} Traditionally, (\ref{1a})-(\ref{1c}) with $X_{j}$, $j=1,\cdots,N$,
following independent, standard uniform distributions, has been identified
as the ANOVA decomposition \cite{sobol03}; however, more
recent works \cite{rahman08,rahman09} reveal no fundamental requirement
for a specific probability measure of $\mathbf{X}$, provided that
the resultant integrals in (\ref{1a})-(\ref{1c}) exist and are finite.
In this work, the ADD should be interpreted with respect to an arbitrary
but product type probability measure, instilling desirable orthogonal properties.
%\end{remark}

\subsubsection{ADD Approximation}

The $S$-variate ADD approximation $\tilde{y}_{S}(\mathbf{X})$, say,
of $y(\mathbf{X})$, where $0\le S<N$, is obtained by truncating
the right side of (\ref{1a}) at $0\le|u|\le S$, yielding
\begin{equation}
\tilde{y}_{S}(\mathbf{X})={\displaystyle \sum_{{\textstyle {u\subseteq\{1,\cdots,N\}\atop 0\le|u|\le S}}}y_{u}(\mathbf{X}_{u})}.\label{3}
\end{equation}
Applying the expectation operator on $y(\mathbf{X})$ and $\tilde{y}_{S}(\mathbf{X})$
from (\ref{1a}) and (\ref{3}), respectively, and noting Proposition
\ref{prop:1}, the mean
\begin{equation}
\mathbb{E}\left[\tilde{y}_{S}(\mathbf{X})\right]=y_{\emptyset}\label{4}
\end{equation}
of the $S$-variate ADD approximation matches the exact mean $\mathbb{E}\left[y(\mathbf{X})\right]:=\int_{\mathbb{R}^{N}}y(\mathbf{x})f_{\mathbf{X}}(\mathbf{x})d\mathbf{x}=y_{\emptyset}$,
regardless of $S$. Applying the expectation operator again, this
time on $\left(\tilde{y}_{S}(\mathbf{X})-y_{\emptyset}\right)^{2}$,
and recognizing Proposition \ref{prop:2} results in splitting the
variance \cite{rahman12,sobol01}
\begin{equation}
\tilde{\sigma}_{S}^{2}:=\mathbb{E}\left[\left(\tilde{y}_{S}(\mathbf{X})-y_{\emptyset}\right)^{2}\right]=\sum_{{\textstyle {\emptyset\ne u\subseteq\{1,\cdots,N\}\atop 1\le|u|\le S}}}\sigma_{u}^{2}=\sum_{s=1}^{S}\:\sum_{{\textstyle {\emptyset\ne u\subseteq\{1,\cdots,N\}\atop |u|=s}}}\sigma_{u}^{2}\label{5}
\end{equation}
of the $S$-variate ADD approximation into variances $\sigma_{u}^{2}:=\mathbb{E}\left[y_{u}^{2}(\mathbf{X}_{u})\right]$,
$\emptyset\ne u\subseteq\{1,\cdots,N\}$, of \emph{zero-}mean ADD
component functions $y_{u}$. Clearly, the approximate variance in
(\ref{5}) approaches the exact variance
\[
\sigma^{2}:=\mathbb{E}\left[\left(y(\mathbf{X})-y_{\emptyset}\right)^{2}\right]=\sum_{\emptyset\ne u\subseteq\{1,\cdots,N\}}\sigma_{u}^{2}=\sum_{s=1}^{N}\:\sum_{{\textstyle {\emptyset\ne u\subseteq\{1,\cdots,N\}\atop |u|=s}}}\sigma_{u}^{2},
\]
the sum of all variance terms, when $S\to N$. A normalized version
$\sigma_{u}^{2}/\sigma^{2}$ is often called the global sensitivity
index of $y$ for $\mathbf{X}_{u}$ \cite{sobol01}.

\subsection{FDD}

Consider a multiplicative form,
\begin{equation}
y(\mathbf{\mathbf{X}})={\displaystyle \prod_{u\subseteq\{1,\cdots,N\}}\left[1+z_{u}(\mathbf{X}_{u})\right]},\label{6}
\end{equation}
of the dimensional decomposition of $y$, where $z_{u}$, $u\subseteq\{1,\cdots,N\}$,
are various component functions of input variables with increasing
dimensions. Like the sum in Equation \ref{1a}, the product in Equation
\ref{6} comprises $2^{N}$ terms, with each term depending on a group
of variables indexed by a particular subset of $\{1,\cdots,N\}$,
including the empty set $\emptyset$. This multiplicative decomposition exists and
is unique for any square-integrable function
$y \in \mathcal{L}_{2}(\Omega,\mathcal{F},P)$ with a non-zero mean.

The FDD, originally proposed by Tunga and Demiralp \cite{tunga05} under the name of factorized high-dimensional model representation, has yet to receive due attention for uncertainty quantification of complex systems.  A prime reason why FDD is not on par with ADD is the lack of explicit relationships between their component functions. Lemma \ref{newlem} and Theorem \ref{thm:1} reveal the desired relationships.

\begin{lemma}
\label{newlem}The FDD component functions, $z_{v}$, $\emptyset \ne v\subseteq\{1,\cdots,N\}$, of a square-integrable function $y:\mathbb{R}^{N}\to\mathbb{R}$ with a non-zero mean $\mathbb{E}[y(\mathbf{X})]=y_{\emptyset} \ne 0$ satisfy
\begin{equation}
1+z_{v}(\mathbf{X}_{v}) \ne 0.
\nonumber
\end{equation}
\end{lemma}
\begin{proof}
Suppose to the contrary that $1+z_{v}(\mathbf{X}_{v}) = 0$ for any $\emptyset \ne v\subseteq\{1,\cdots,N\}$. Then it follows from (\ref{6}) that $y(\mathbf{X})=0$ and hence $\mathbb{E}[y(\mathbf{X})]=0$.  This contradicts the assumption that $\mathbb{E}[y(\mathbf{X})]=y_{\emptyset} \ne 0$, completing the proof.
\end{proof}

\begin{theorem}
\label{thm:1}The recursive relationships between component functions
of the ADD and FDD of a non-zero mean, square-integrable function $y:\mathbb{R}^{N}\to\mathbb{R}$,
represented by (\ref{1a}) and (\ref{6}), respectively, are
\begin{equation}
1+z_{u}(\mathbf{X}_{u})=\frac{{\displaystyle \sum_{v\subseteq u}y_{v}(\mathbf{X}_{v})}}{{\displaystyle \prod_{v\subset u}\left[1+z_{v}(\mathbf{X}_{v})\right]}},\; u\subseteq\{1,\cdots,N\}.\label{7}
\end{equation}
\end{theorem}
\begin{proof}
Since (\ref{1a}) and (\ref{6}) represent the same function $y$,
\begin{equation}
{\displaystyle \sum_{u\subseteq\{1,\cdots,N\}}y_{u}(\mathbf{X}_{u})}=\prod_{u\subseteq\{1,\cdots,N\}}\left[1+z_{u}(\mathbf{X}_{u})\right],\label{8}
\end{equation}
which, as is, is unwieldy to solve for $z_{u}$. Instead, expand the
right side of (\ref{6}) to form
\begin{equation}
\begin{split}{\displaystyle y(\mathbf{X})} & =1+z_{\emptyset}{\displaystyle +{\displaystyle \sum_{{\textstyle {u\subseteq\{1,\cdots,N\}\atop |u|=1}}}}r_{u}\left(z_{v}(\mathbf{X}_{v});v\subseteq u\right)}+{\displaystyle \sum_{{\textstyle {u\subseteq\{1,\cdots,N\}\atop |u|=2}}}}r_{u}\left(z_{v}(\mathbf{X}_{v});v\subseteq u\right)+\cdots+\\
 & \;\;\;\;\; r_{\{1,\cdots,N\}}\left(z_{v}(\mathbf{X}_{v});v\subseteq\{1,\cdots,N\}\right)\\
 & =\sum_{u\subseteq\{1,\cdots,N\}}r_{u}\left(z_{v}(\mathbf{X}_{v});v\subseteq u\right),
\end{split}
\label{9}
\end{equation}
where $r_{u}\left(z_{v}(\mathbf{X}_{v});v\subseteq u\right)$ is a
function of at most $|u|$-variate multiplicative component functions
of $y$. For instance, when $u=\emptyset$, $u=\{i\}$, and $u=\{i_{1},i_{2}\}$,
$i,i_{1},i_{2}=1,\cdots,N$, $i_{2}>i_{1}$, the corresponding $r_{u}$-functions
are $r_{\emptyset}(z_{\emptyset})=1+z_{\emptyset}$, $r_{\{i\}}(z_{\emptyset},z_{\{i\}}(X_{i}))$,
and $r_{\{i_{1},i_{2}\}}(z_{\emptyset},z_{\{i_{1}\}}(X_{i_{1}}),z_{\{i_{2}\}}(X_{i_{2}}),z_{\{i_{1},i_{2}\}}(X_{i_{1}},X_{i_{2}}))$,
respectively. Comparing (\ref{1a}) and (\ref{9}) yields the recursive
relationship,
\begin{equation}
r_{u}\left(z_{v}(\mathbf{X}_{v});v\subseteq u\right)=y_{u}(\mathbf{X}_{u}),\label{10}
\end{equation}
which, on inversion, expresses $z_{u}$, $u\subseteq\{1,\cdots N\}$,
in terms of the additive ANOVA component functions $y_{v}$, $v\subseteq u$.
Therefore, all remaining additive or multiplicative component functions of (\ref{8})
not involved can be ignored. Indeed, setting
$y_{v}=z_{v}=0$ for all $v\nsubseteq u$ in (\ref{8}) results in
\begin{equation}
{\displaystyle \sum_{v\subseteq u}y_{v}(\mathbf{X}_{v})}=\prod_{v\subseteq u}\left[1+z_{v}(\mathbf{X}_{v})\right]=\left[1+z_{u}(\mathbf{X}_{u})\right]\prod_{v\subset u}\left[1+z_{v}(\mathbf{X}_{v})\right],\label{11}
\end{equation}
where Lemma \ref{newlem} assures that $1+z_{v}(\mathbf{X}_{v})\ne 0$ for any $\emptyset\ne v \subseteq u$.  On inversion, (\ref{11}) yields (\ref{7}), completing the proof.~\end{proof}
\begin{corollary}
\label{cor:1}Recursive evaluations of (\ref{7}) eliminate $1+z_{v}$,
$v\subset u$, leading to an explicit form of
\begin{equation}
1+z_{u}(\mathbf{X}_{u})=\cfrac{{\displaystyle \sum_{w_{|u|}\subseteq u}y_{w_{|u|}}\left(\mathbf{X}_{w_{|u|}}\right)}}{{\displaystyle \prod_{w_{|u|}\subset u}}\cfrac{{\displaystyle \sum_{w_{|u|-1}\subseteq w_{|u|}}y_{w_{|u|-1}}\left(\mathbf{X}_{w_{|u|-1}}\right)}}{{\displaystyle \prod_{w_{|u|-1}\subset w_{|u|}}}\cfrac{\ddots}{\ddots{\displaystyle \prod_{w_{2}\subset w_{3}}}\cfrac{{\displaystyle \sum_{w_{1}\subseteq w_{2}}}y_{w_{1}}\left(\mathbf{X}_{w_{1}}\right)}{{\displaystyle \prod_{w_{1}\subset w_{2}}}\cfrac{{\displaystyle \sum_{w_{0}\subseteq w_{1}}}y_{w_{0}}\left(\mathbf{X}_{w_{0}}\right)}{{\displaystyle 1}}}}}}\label{12}
\end{equation}
for any $u\subseteq\{1,\cdots,N\}$, solely in terms of the ANOVA
component functions.
\end{corollary}
\,
\begin{corollary}
\label{cor:2}The multiplicative constant, univariate, and bivariate
component functions of a square-integrable function $y:\mathbb{R}^{N}\to\mathbb{R}$,
obtained by setting $u=\emptyset$, $u=\{i\};\: i=1,\cdots,N$, and
$u=\{i_{1},i_{2}\};\: i_{1}<i_{2}=1,\cdots,N$, respectively, in (\ref{7})
or (\ref{12}) are
\begin{equation}
{\displaystyle 1+z_{\emptyset}}={\displaystyle y_{\emptyset}},\label{13}
\end{equation}
\begin{equation}
1+{\displaystyle z_{\{i\}}(X_{i})}=\frac{{\displaystyle {\displaystyle {\displaystyle y_{\emptyset}}+y_{\{i\}}(X_{i})}}}{{\displaystyle y_{\emptyset}}},\label{14}
\end{equation}
and
\begin{equation}
1+{\displaystyle z_{\{i_{1},i_{2}\}}(X_{i_{1}},X_{i_{2}})}=\frac{{\displaystyle y_{\emptyset}}+y_{\{i_{1}\}}(X_{i_{1}})+y_{\{i_{2}\}}(X_{i_{2}})+{\displaystyle {\displaystyle y_{\{i_{1},i_{2}\}}(X_{i_{1}},X_{i_{2}})}}}{{\displaystyle {\displaystyle y_{\emptyset}}}{\displaystyle \left[\frac{{\displaystyle y_{\emptyset}}+y_{\{i_{1}\}}(X_{i_{1}})}{{\displaystyle y_{\emptyset}}}\right]}{\displaystyle \left[\frac{{\displaystyle y_{\emptyset}}+y_{\{i_{2}\}}(X_{i_{2}})}{{\displaystyle y_{\emptyset}}}\right]}}.\label{15}
\end{equation}

\end{corollary}

%\begin{remark}
{\em Remark 2.} Equations (\ref{13}), (\ref{14}), and (\ref{15}) can also
be obtained employing the identity and first- and second-degree idempotent
operators \cite{tunga05}. However, to obtain similar expressions
for trivariate and higher-variate multiplicative component functions,
an extensive amount of algebra associated with third- and higher-degree
idempotent operators will be required. This is a primary reason why
component functions with three or more variables have yet to be reported
in the current literature. Theorem \ref{thm:1}, in contrast, is simpler
and, more importantly, provides a general expression $-$ (\ref{7})
or (\ref{12}) $-$ that is valid for a multiplicative component function
of an arbitrary number of variables.
%\end{remark}

%\begin{remark}
{\em Remark 3.} When $y_{\emptyset}$ is\emph{ zero} or is close to \emph{zero},
(\ref{6}) through (\ref{15}) may fail or become ill-conditioned,
raising questions about the suitability of an FDD approximation in
such conditions. However, they do not necessarily imply that the FDD
cannot be used. Indeed, all of these problems can be remedied by appropriately
conditioning the function $y$. For instance, by adding a non-zero
constant to $y$ or multiplying $y$ with a non-zero constant,
(\ref{6}) through (\ref{15}) for the pre-conditioned $y$ remain
valid and well-behaved.
%\end{remark}

\subsubsection{FDD Approximation}

The $S$-variate FDD approximation $\hat{y}_{S}(\mathbf{X})$, say,
of $y(\mathbf{X})$, where $0\le S<N$, is obtained by truncating
the right side of (\ref{6}) at $0\le|u|\le S$, yielding
\begin{equation}
\hat{y}_{S}(\mathbf{X})={\displaystyle {\displaystyle \prod_{{\textstyle {u\subseteq\{1,\cdots,N\}\atop 0\le|u|\le S}}}}\left[1+z_{u}(\mathbf{X}_{u})\right]}.\label{16}
\end{equation}
It is important to note that the right side of (\ref{16}) contains products of at most $S$-dimensional ANOVA component functions of $y$. Therefore, the term "$S$-variate" used for the FDD approximation should be interpreted in the context of including at most $S$-degree interaction of input variables, even though $\hat{y}_{S}(\mathbf{X})$ is not strictly an $S$-variate function.

Unlike (\ref{4}) and (\ref{5}), the mean and variance of $\hat{y}_{S}(\mathbf{X})$,
respectively defined as
\begin{equation}
\mathbb{E}\left[\hat{y}_{S}(\mathbf{X})\right]:=\int_{\mathbb{R}^{N}}\hat{y}_{S}(\mathbf{x})f_{\mathbf{X}}(\mathbf{x})d\mathbf{x}\label{17}
\end{equation}
and
\begin{equation}
\hat{\sigma}_{S}^{2}:=\mathbb{E}\left[\left(\hat{y}_{S}(\mathbf{X})-\mathbb{E}\left[\hat{y}_{S}(\mathbf{X})\right]\right)^{2}\right]:=\int_{\mathbb{R}^{N}}\left(\hat{y}_{S}(\mathbf{x})-\mathbb{E}\left[\hat{y}_{S}(\mathbf{X})\right]\right)^{2}f_{\mathbf{X}}(\mathbf{x})d\mathbf{x},\label{18}
\end{equation}
do not produce closed-form or analytic expressions in
terms of $y_{\emptyset}$ and $\sigma_{u}^{2}$ if $S$ is selected
arbitrarily. This is one drawback of FDD when compared with ADD.  Having said so,
they are easily estimated by sampling methods, such as quasi and crude Monte Carlo simulations, or even numerical integration if $N$ is
not overly large. Not wishing to preempt subsequent discussion, these two moments can be obtained in closed-form when $S=1$, that is, for the univariate
FDD approximation, to be explained in Section 5.

Given a truncation $S$, which approximation stemming from ADD and
FDD is more precise? The answer depends on the class of multivariate
function being approximated. Indeed, the relative advantage or disadvantage
of ADD and FDD should be judged by examining three special classes of
functions: purely multiplicative functions, purely additive functions,
and their mixtures, described as follows.

\subsection{Purely Multiplicative Functions}

Consider a function of pure multiplicative form
\begin{equation}
y(\mathbf{X})=\nu_{0}{\displaystyle {\displaystyle \prod_{i=1}^{N}h_{i}(X_{i})},}\label{19}
\end{equation}
where $\nu_{0}\in\mathbb{R}$ is a constant and $h_{i}:\mathbb{R}\to\mathbb{R}$
, $i=1,\cdots,N$, are square-integrable univariate functions with
\begin{equation}
\nu_{i}:=\mathbb{E}\left[h_{i}(X_{i})\right]:=\int_{\mathbb{R}}h_{i}(x_{i})f_{i}(x_{i})dx_{i},
\nonumber
\end{equation}
\begin{equation}
\delta_{i}^{2}:=\mathbb{E}\left[\left(h_{i}(X_{i})-\nu_{i}\right)^{2}\right]:=\int_{\mathbb{R}}\left(h_{i}(x_{i})-\nu_{i}\right)^{2}f_{i}(x_{i})dx_{i}
\nonumber
\end{equation}
denoting their first two moments. For a function $y$ satisfying (\ref{19}),
both ADD and FDD component functions and their respective variances
can be determined with little effort. Indeed, when (\ref{1b}) and (\ref{1c}) are used, the outcome is
\begin{equation}
y_{\emptyset}=1+z_{\emptyset}=\nu_{0}{\displaystyle {\displaystyle \prod_{i=1}^{N}\nu_{i}}},\label{22}
\end{equation}
\begin{equation}
y_{u}(\mathbf{X}_{u})=\nu_{0}\left({\displaystyle {\displaystyle \prod_{i\notin u}\nu_{i}}}\right){\displaystyle {\displaystyle \prod_{i\in u}\left[h_{i}(X_{i})-\nu_{i}\right]}},\;\emptyset\ne u\subseteq\{1,\cdots,N\},\label{23}
\end{equation}
whereas applying (\ref{22}) and (\ref{23}) to (\ref{7}) or (\ref{12}) results in
\begin{equation}
 1+z_{u}(\mathbf{X}_{u})=
  \begin{cases}
  {\displaystyle \frac{h_{i}(X_{i})}{\nu_{i}}} & \text{if}\: u=\{i\},\: i=1,\cdots,N,\\
   1       & \text{if } |u|\ge2.
  \end{cases}
\label{24}
\end{equation}
The ADD component functions lead to the $S$-variate variance
approximation
\begin{equation}
\tilde{\sigma}_{S}^{2}=\sum_{s=1}^{S}\:\sum_{{\textstyle {\emptyset\ne u\subseteq\{1,\cdots,N\}\atop |u|=s}}}\nu_{0}^{2}\left({\displaystyle {\displaystyle \prod_{i\notin u}\nu_{i}^{2}}}\right){\displaystyle {\displaystyle \prod_{i\in u}\delta_{i}^{2}}}\label{25}
\end{equation}
that approaches the exact variance
\begin{equation}
\sigma^{2}=\nu_{0}^{2}\left[{\displaystyle {\displaystyle \prod_{i=1}^{N}\left(\delta_{i}^{2}+\nu_{i}^{2}\right)}}-{\displaystyle {\displaystyle \prod_{i=1}^{N}\nu_{i}^{2}}}\right]\label{26}
\end{equation}
of $y$ as $S\to N$. In contrast, the FDD component functions, due
to their multiplicative structure, yield the exact variance, that is,
$\hat{\sigma}_{S}^{2}=\sigma^{2}$ for any $1\le S\le N$. The implication is that the univariate truncation of FDD is adequate for calculating
the variance or any other probabilistic characteristics of a purely multiplicative function $y$. This
is because such a function is exactly reproduced
by its univariate FDD approximation. The same does not hold true
for an ADD approximation, requiring, therefore, higher-variate truncations to
approximate $y$ progressively accurately. Clearly, FDD is more relevant
than ADD for the multiplicative class of functions.

\subsection{Purely Additive Functions}

Now consider a function of purely additive form
\begin{equation}
y(\mathbf{X})=\mu_{0}+\sum_{i=1}^{N}g_{i}(X_{i}).\label{29}
\end{equation}
Here $\mu_{0}\in\mathbb{R}$ is another constant and $g_{i}:\mathbb{R}\to\mathbb{R}$
, $i=1,\cdots,N$, are square-integrable univariate functions with
\begin{equation}
\mu_{i}:=\mathbb{E}\left[g_{i}(X_{i})\right]:=\int_{\mathbb{R}}g_{i}(x_{i})f_{i}(x_{i})dx_{i},
\nonumber
\end{equation}
\begin{equation}
\lambda_{i}^{2}:=\mathbb{E}\left[\left(g_{i}(X_{i})-\mu_{i}\right)^{2}\right]:=\int_{\mathbb{R}}\left(g_{i}(x_{i})-\mu_{i}\right)^{2}f_{i}(x_{i})dx_{i}
\nonumber
\end{equation}
designating their first two moments. Employing (\ref{1a})-(\ref{1c})
and (\ref{7}), the ADD and FDD component functions of $y$ in (\ref{29}) are
\begin{equation}
y_{\emptyset}=1+z_{\emptyset}=\mu_{0}+\sum_{i=1}^{N}\mu_{i},\label{30}
\end{equation}
\begin{equation}
 y_{u}(\mathbf{X}_{u})=
  \begin{cases}
  {\displaystyle g_{i}(X_{i})-\mu_{i}} & \text{if}\: u=\{i\},\: i=1,\cdots,N,\\
   0                                           & \text{if } |u|\ge2,
  \end{cases}
\label{31}
\end{equation}
and
\begin{equation}
1+z_{u}(\mathbf{X}_{u})=\dfrac{\mu_{0}+ \left( {\displaystyle \sum_{i=1}^{N}\mu_{i}} \right) +
{\displaystyle \sum_{i\in u}\left[g_{i}(X_{i})-\mu_{i}\right]}}
{{\displaystyle \prod_{v\subset u}\left[1+z_{v}(\mathbf{X}_{v})\right]}},\label{32}
\end{equation}
with the non-constant FDD component functions expressed recursively.
In this case, the variance due to an $S$-variate ADD approximation
yields the exact variance
\begin{equation}
\tilde{\sigma}_{S}^{2}=\sigma^{2}=\sum_{i=1}^{N}\lambda_{i}^{2}\label{33}
\end{equation}
for any $1\le S\le N$. However, due to the complicated form of (\ref{32}),
an explicit formula for the variance of a generic $S$-variate FDD
approximation is not possible. But given a truncation $S$, the corresponding
variance can be estimated by sampling methods or numerical integration.
Nonetheless, if a function has a purely additive structure, the univariate
ADD approximation exactly reproduces that function, thereby needing
at most the univariate truncation of ADD. In contrast, FDD will now require
higher-variate truncations for rendering gradually accurate approximations.
For purely additive functions, ADD is, therefore, more appropriate
than FDD, turning the table on the latter decomposition.

\subsection{Mixtures of Multiplicative and Additive Functions}

Finally, consider a blended function of the form
\begin{equation}
y(\mathbf{X})=\nu_{0}\left( {\displaystyle {\displaystyle \prod_{i=1}^{N}h_{i}(X_{i})}}\right) +\mu_{0}+\sum_{i=1}^{N}g_{i}(X_{i}),\label{32b}
\end{equation}
built on adding the purely multiplicative and purely additive pieces from (\ref{19}) and (\ref{29}). Depending on relative orders of these constituents, $y$ can be dominantly multiplicative or dominantly additive or neither.
Regardless, (\ref{1a})-(\ref{1c}) result in its ADD component functions
\begin{equation}
y_{\emptyset}=\nu_{0}\left( {\displaystyle {\displaystyle \prod_{i=1}^{N}\nu_{i}}}\right) +\mu_{0}+\sum_{i=1}^{N}\mu_{i},\label{32c}
\end{equation}
\begin{equation}
 y_{u}(\mathbf{X}_{u})=
  \begin{cases}
  {\displaystyle \nu_{0}\left( {\displaystyle {\displaystyle \prod_{j=1,j\neq i}^{N}\nu_{j}}}\right) {\displaystyle {\displaystyle \left[h_{i}(X_{i})-\nu_{i}\right]}+g_{i}(X_{i})-\mu_{i}}} & \text{if } u=\{i\}, i=1,\cdots,N,\\
\nu_{0}\left( {\displaystyle {\displaystyle \prod_{i\notin u}\nu_{i}}}\right) {\displaystyle {\displaystyle \prod_{i\in u}\left[h_{i}(X_{i})-\nu_{i}\right]}} & \text{if }|u|\ge2,
  \end{cases}
\label{32d}
\end{equation}
lending themselves to calculate the variance
\begin{equation}
 \tilde{\sigma}_{S}^{2}=
  \begin{cases}
 {\displaystyle \sum_{i=1}^{N}}\left[\nu_{0}^{2}\left( {\displaystyle {\displaystyle \prod_{j=1,j\neq i}^{N}\nu_{j}^{2}}}\right) {\displaystyle {\displaystyle \delta_{i}^{2}}}+\lambda_{i}^{2}+2\nu_{0}\left( {\displaystyle {\displaystyle \prod_{j=1,j\neq i}^{N}\nu_{j}}}\right) {\displaystyle {\displaystyle \eta_{i}^{2}}}\right] & \text{if } S=1,\\
\tilde{\sigma}_{1}^{2}+{\displaystyle \sum_{s=2}^{S}\:\sum_{{\textstyle {\emptyset\ne u\subseteq\{1,\cdots,N\}\atop |u|=s}}}}\nu_{0}^{2}\left( {\displaystyle {\displaystyle \prod_{i\notin u}\nu_{i}^{2}}}\right) {\displaystyle {\displaystyle \prod_{i\in u}\delta_{i}^{2}}} & \text{if } S\ge2,
  \end{cases}
\label{32e}
\end{equation}
of an $S$-variate ADD approximation, where
\begin{equation}
\eta_i^2:=\mathbb{E}\left[\left(h_{i}(X_{i})-\nu_{i}\right)\left(g_{i}(X_{i})-\mu_{i}\right)\right]:=
\int_{\mathbb{R}}\left(h_{i}(x_{i})-\nu_{i}\right)\left(g_{i}(x_{i})-\mu_{i}\right)f_{i}(x_{i})dx_{i}
\nonumber
\end{equation}
is the covariance between  $h_i(X_i)$ and $g_i(X_i)$. The approximate variance approaches the exact variance
\begin{equation}
\sigma^{2} = \nu_{0}^{2}\left[{\displaystyle {\displaystyle \prod_{i=1}^{N}\left(\delta_{i}^{2}+\nu_{i}^{2}\right)}}-{\displaystyle {\displaystyle \prod_{i=1}^{N}\nu_{i}^{2}}}\right]+\sum_{i=1}^{N}\lambda_{i}^{2}
+ 2\nu_{0}\left({\displaystyle {\displaystyle \prod_{i=1}^{N}\nu_{i}}}\right)
\sum_{i=1}^{N} {\displaystyle \frac{\eta_{i}^{2}}{\nu_{i}}},
\label{32f}
\end{equation}
where the first two terms are the variances of the purely multiplicative
and purely additive pieces, while the third term represents the covariance
between these two pieces. The formulae for $y_{u}(\mathbf{X}_{u})$,
$\tilde{\sigma}_{S}^{2}$, and $\sigma^{2}$ given in (\ref{32c}) through
(\ref{32f}) are slightly general, shrinking to
(\ref{22}) through (\ref{26}) for purely multiplicative
functions and to (\ref{30}) through (\ref{33}) for purely additive functions,
as expected. Although the FDD component functions can be easily derived
as per (\ref{7}), again, an explicit formula for the variance $\hat{\sigma}_{S}^{2}$
of FDD approximation remains impalpable for a generic truncation $S$. Unlike the two pure function classes discussed in preceding subsections, it is not obvious which approximation between ADD and FDD is better for this mixed class of functions.

\subsection{Example 1}

Consider two functions,
\begin{equation}
y_{1}=\nu_{0}{\displaystyle {\displaystyle \prod_{i=1}^{N}X_{i}}}\;\text{and}\; y_{2}=\mu_{0}+\sum_{i=1}^{N}X_{i},\label{34}
\end{equation}
endowed with a purely multiplicative and a purely additive structure,
respectively, where $X_{i},$ $i=1,\cdots,N$, are $N$ independent
and identically distributed random variables, each following uniform
distribution over {[}0,1{]}, $\nu_{0}=100$, and $\mu_{0}=0$. The
mean and variance of $y_{1}(\mathbf{X})$ are $\nu_{0}/2^{N}$ and
$\nu_{0}^{2}/3^{N}-\nu_{0}^{2}/2^{2N}$, respectively, and of $y_{2}(\mathbf{X})$
are $\mu_{0}+N/2$ and $N/12$, respectively. The ADD and FDD component
functions of $y_{1}$ or $y_{2}$ were obtained using (\ref{22})
through (\ref{24}) or (\ref{30}) through (\ref{32}). The means
and variances of ADD approximations were calculated exactly using
(\ref{22}) or (\ref{30}) and (\ref{25}) or (\ref{33}), whereas
the means and variances of FDD approximations {[}(\ref{17}), (\ref{18}){]}
were estimated by a fully symmetric multidimensional integration rule
with nine generators \cite{genz83}. The purpose of this example is
to compare the variances of both functions obtained using various
ADD and FDD approximations.

Tables \ref{table1} and \ref{table2} respectively present the mean-squared errors by univariate ($S=1$) to decavariate ($S=10$) truncations of ADD for $y_{1}$
and by univariate ($S=1$) to pentavariate ($S=5$) truncations of
FDD for $y_{2}$ when $N=6,7,8,9,10$. The error is defined as the
absolute difference between the exact ($\sigma^{2}$) and approximate
($\tilde{\sigma}_{S}^{2}$ or $\hat{\sigma}_{S}^{2}$) variances, divided by the exact variance.
Given a problem size ($N$), the errors decay with increasing $S$
for both functions as expected, although the ADD approximations for
a fixed $S$ worsen when $N$ grows larger. This implies that
an ADD approximation becomes less precise for high-dimensional multiplicative
functions. No such degradations are observed for FDD approximations of high-dimensional additive functions.  Both errors vanish altogether when $S=N=10$, as the decavariate
ADD and FDD approximations coincide with $y_{1}$ and $y_{2}$, respectively.
Due to the purely multiplicative and purely additive functional forms
examined, their univariate FDD and univariate ADD approximations are
the same as $y_{1}$ and $y_{2}$, producing respective variances
exactly.

\begin{table}
\caption{ADD approximation errors in variances of $y_{1}$ in Example 1$^{(\text{a})}$ }
\begin{centering}
\begin{tabular}{cccccc}
\hline
 & \multicolumn{5}{c}{{\footnotesize $N$}}\tabularnewline
\cline{2-6}
{\footnotesize $S$} & {\footnotesize 6} & {\footnotesize 7} & {\footnotesize 8} & {\footnotesize 9} & {\footnotesize 10}\tabularnewline
\hline
{\footnotesize 1} & {\footnotesize 0.566974} & {\footnotesize 0.640558} & {\footnotesize 0.703332} & {\footnotesize 0.75646} & {\footnotesize 0.801087}\tabularnewline
{\footnotesize 2} & {\footnotesize 0.206118} & {\footnotesize 0.281116} & {\footnotesize 0.357219} & {\footnotesize 0.43174} & {\footnotesize 0.502717}\tabularnewline
{\footnotesize 3} & {\footnotesize $4.5738\times10^{-2}$} & {\footnotesize $8.1426\times10^{-2}$} & {\footnotesize 0.126477} & {\footnotesize 0.179179} & {\footnotesize 0.237499}\tabularnewline
{\footnotesize 4} & {\footnotesize $5.6430\times10^{-3}$} & {\footnotesize $1.4862\times10^{-2}$} & {\footnotesize $3.0335\times10^{-2}$} & {\footnotesize $5.2899\times10^{-2}$} & {\footnotesize $8.2789\times10^{-2}$}\tabularnewline
{\footnotesize 5} & {\footnotesize $2.9700\times10^{-4}$} & {\footnotesize $1.5496\times10^{-3}$} & {\footnotesize $4.6969\times10^{-3}$} & {\footnotesize $1.0806\times10^{-2}$} & {\footnotesize $2.0905\times10^{-2}$}\tabularnewline
{\footnotesize 6} & {\footnotesize 0} & {\footnotesize $7.0437\times10^{-5}$} & {\footnotesize $4.2391\times10^{-4}$} & {\footnotesize $1.4518\times10^{-3}$} & {\footnotesize $3.7149\times10^{-3}$}\tabularnewline
{\footnotesize 7} & {\footnotesize $-^{(\text{b})}$} & {\footnotesize 0} & {\footnotesize $1.6956\times10^{-5}$} & {\footnotesize $1.1548\times10^{-4}$} & {\footnotesize $4.4062\times10^{-4}$}\tabularnewline
{\footnotesize 8} & {\footnotesize $-^{(\text{b})}$} & {\footnotesize $-^{(\text{b})}$} & {\footnotesize 0} & {\footnotesize $4.1244\times10^{-6}$} & {\footnotesize $3.1328\times10^{-5}$}\tabularnewline
{\footnotesize 9} & {\footnotesize $-^{(\text{b})}$} & {\footnotesize $-^{(\text{b})}$} & {\footnotesize $-^{(\text{b})}$} & {\footnotesize 0} & {\footnotesize $1.0106\times10^{-6}$}\tabularnewline
{\footnotesize 10} & {\footnotesize $-^{(\text{b})}$} & {\footnotesize $-^{(\text{b})}$} & {\footnotesize $-^{(\text{b})}$} & {\footnotesize $-^{(\text{b})}$} & {\footnotesize 0}\tabularnewline
\hline
\end{tabular}
\par\end{centering}
{\footnotesize ~~~~~~(a) The error is defined as the absolute
difference between exact and approximate variances,}{\footnotesize \par}
{\footnotesize ~~~~~~~~~~~divided by the exact variance.}{\footnotesize \par}
{\footnotesize ~~~~~~(b) Not applicable.}
\label{table1}
\end{table}

\begin{table}
\caption{FDD approximation errors in variances of $y_{2}$ in Example 1$^{(\text{a})}$ }
\begin{centering}
\begin{tabular}{cccccc}
\hline
 & \multicolumn{5}{c}{{\footnotesize $N$}}\tabularnewline
\cline{2-6}
{\footnotesize $S$} & {\footnotesize 6} & {\footnotesize 7} & {\footnotesize 8} & {\footnotesize 9} & {\footnotesize 10}\tabularnewline
\hline
{\footnotesize 1} & {\footnotesize $2.3436\times10^{-2}$} & {\footnotesize $2.0641\times10^{-2}$} & {\footnotesize $1.8420\times10^{-2}$} & {\footnotesize $1.6620\times10^{-2}$} & {\footnotesize $1.5134\times10^{-2}$}\tabularnewline
{\footnotesize 2} & {\footnotesize $1.1620\times10^{-3}$} & {\footnotesize $9.2133\times10^{-4}$} & {\footnotesize $7.4728\times10^{-4}$} & {\footnotesize $6.1788\times10^{-4}$} & {\footnotesize $5.1923\times10^{-4}$}\tabularnewline
{\footnotesize 3} & {\footnotesize $1.0918\times10^{-4}$} & {\footnotesize $7.8170\times10^{-5}$} & {\footnotesize $5.5198\times10^{-5}$} & {\footnotesize $4.2094\times10^{-5}$} & {\footnotesize $3.2728\times10^{-5}$}\tabularnewline
{\footnotesize 4} & {\footnotesize $1.7728\times10^{-5}$} & {\footnotesize $1.3099\times10^{-5}$} & {\footnotesize $9.3640\times10^{-6}$} & {\footnotesize $6.7126\times10^{-6}$} & {\footnotesize $4.8768\times10^{-6}$}\tabularnewline
{\footnotesize 5} & {\footnotesize $2.5480\times10^{-6}$} & {\footnotesize $2.5579\times10^{-6}$} & {\footnotesize $1.9704\times10^{-6}$} & {\footnotesize $1.4152\times10^{-6}$} & {\footnotesize $9.9971\times10^{-7}$}\tabularnewline
\hline
\end{tabular}
\par\end{centering}
{\footnotesize ~~~~~~(a) The error is defined as the absolute
difference between exact and approximate variances,}{\footnotesize \par}
{\footnotesize ~~~~~~~~~~~divided by the exact variance.}
\label{table2}
\end{table}

\section{Hybrid Dimensional Decomposition}

When a desired stochastic response exhibits neither a dominantly additive
nor a dominantly multiplicative nature, then a mixed approach that
optimally combines ADD and FDD approximations is needed. Three hybrid
approximations, employing a nonlinear and two linear mixtures of ADD and
FDD, are proposed.

\subsection{Hybrid Approximations}

Given $S$-variate ADD and FDD approximations $\tilde{y}_{S}(\mathbf{X})$
and $\hat{y}_{S}(\mathbf{X})$, let
\begin{equation}
 \bar{y}_{S}(\mathbf{X};\alpha_{S},\beta_{S},\gamma_{S},\cdots):=
  \begin{cases}
  y_{\emptyset}                                                                               & \text{if } S=0,         \\
  h(\tilde{y}_{S}(\mathbf{X}),\hat{y}_{S}(\mathbf{X});\alpha_{S},\beta_{S},\gamma_{S},\cdots) & \text{if } 1 \le S < N, \\
   y(\mathbf{X})                                                                              & \text{if } S=N,
  \end{cases}
\nonumber
\end{equation}
define a general, $S$-variate hybrid approximation of $y(\mathbf{X})$,
where $h$ is a chosen linear or nonlinear model function such that
$\mathbb{E}[\bar{y}_{S}(\mathbf{X};\alpha_{S},\beta_{S},\gamma_{S},\cdots)]=y_{\emptyset}$
and $\alpha_{S},\beta_{S},\gamma_{S},\cdots$ are the associated model parameters. With no loss of generality, the \emph{zero}-mean functions, defined by
$w(\mathbf{X}):=y(\mathbf{X})-y_{\emptyset}$, $\tilde{w}_{S}(\mathbf{X}):=
\tilde{y}_{S}(\mathbf{X})-y_{\emptyset}$,
$\hat{w}_{S}(\mathbf{X}):=\hat{y}_{S}(\mathbf{X})-\mathbb{E}[\hat{y}_{S}(\mathbf{X})]$,
and $\bar{w}_{S}(\mathbf{X};\alpha_{S},\beta_{S},\gamma_{S},\cdots):=
\bar{y}_{S}(\mathbf{X};\alpha_{S},\beta_{S},\gamma_{S},\cdots)-y_{\emptyset}$,
will be used for the remainder of this section.  It is important to note that the
expectation of the univariate FDD approximation $\hat{y}_{1}(\mathbf{X})$
is ${\displaystyle y_{\emptyset}}$, which is also the mean of $y(\mathbf{X})$.
Since a higher-variate approximation cannot be worse than a lower-variate
approximation, one may conjecture, in absence of a rigorous proof,
that the bivariate and higher-variate FDD approximations also yield
the exact mean. However, such conjecture is neither required nor used for hybrid approximations developed here.  Theorem \ref{thm:2}
and Corollaries \ref{cor:3} and \ref{cor:4} describe three optimal
hybrid approximations $\bar{y}_{S,n}(\mathbf{X};\alpha_{S,n},\beta_{S,n},\gamma_{S,n})$,
$\bar{y}_{S,l}(\mathbf{X};\alpha_{S,l},\beta_{S,l})$, and $\bar{y}_{S,l'}(\mathbf{X};\alpha_{S,l'})$ for $1 \le S < N$, each producing the exact mean $y_{\emptyset}$, where the subscripts $n$, $l$, and $l'$ refer to a nonlinear and two linear models examined in this work.  Their zero-mean counterparts are defined as $\bar{w}_{S,n}(\mathbf{X};\alpha_{S,n},\beta_{S,n},\gamma_{S,n}):=\bar{y}_{S,n}(\mathbf{X};\alpha_{S,n},\beta_{S,n},\gamma_{S,n})-y_{\emptyset}$,
$\bar{w}_{S,l}(\mathbf{X};\alpha_{S,l},\beta_{S,l}):=\bar{y}_{S,l}(\mathbf{X};\alpha_{S,l},\beta_{S,l})-y_{\emptyset}$,
and $\bar{w}_{S,l'}(\mathbf{X};\alpha_{S,l'}):=\bar{y}_{S,l'}(\mathbf{X};\alpha_{S,l'})-y_{\emptyset}$,
respectively.

\begin{theorem}
\label{thm:2} Given an integer $1\le S<N<\infty$, let $\tilde{w}_{S}(\mathbf{X})$
and $\hat{w}_{S}(\mathbf{X})$ represent zero-mean, $S$-variate ADD
and FDD approximations with variances $\tilde{\sigma}_{S}^{2}:=\mathbb{E}[\tilde{y}_{S}(\mathbf{X})-y_{\emptyset}]^{2}=\mathbb{E}[\tilde{w}_{S}^{2}(\mathbf{X})]$
and $\hat{\sigma}_{S}^{2}:=\mathbb{E}[\hat{y}_{S}(\mathbf{X})-\mathbb{E}[\hat{y}_{S}(\mathbf{X})]]^{2}=\mathbb{E}[\hat{w}_{S}^{2}(\mathbf{X})]$,
respectively, of a real-valued, zero-mean, square-integrable function
$w(\mathbf{X})$. Then an optimal, nonlinear, $S$-variate
hybrid approximation of $w(\mathbf{X})$, if it exists, is
\begin{equation}
\begin{array}{rcl}
\bar{w}_{S,n}(\mathbf{X};\alpha_{S,n},\beta_{S,n},\gamma_{S,n}) & = & \alpha_{S,n}\tilde{w}_{S}(\mathbf{X})+\beta_{S,n}\hat{w}_{S}(\mathbf{X})\\
                                                                &   &
+\gamma_{S,n}\left[\tilde{w}_{S}(\mathbf{X})\hat{w}_{S}(\mathbf{X})-\mathbb{E}\left\{ \tilde{w}_{S}(\mathbf{X})\hat{w}_{S}(\mathbf{X})\right\} \right],
\end{array}
\label{36}
\end{equation}
where $(\alpha_{S,n},\beta_{S,n},\gamma_{S,n})\in\mathbb{R}^{3}$
is the solution of
\begin{equation}
\begin{array}{l}
\left[\!\!\begin{array}{ccc}
\tilde{\sigma}_{S}^{2} & \mathbb{E}\left[\tilde{w}_{S}(\mathbf{X})\hat{w}_{S}(\mathbf{X})\right] & \mathbb{E}\left[\tilde{w}_{S}^{2}(\mathbf{X})\hat{w}_{S}(\mathbf{X})\right]\\
 & \hat{\sigma}_{S}^{2} & \mathbb{E}\left[\tilde{w}_{S}(\mathbf{X})\hat{w}_{S}^{2}(\mathbf{X})\right]\\
(\text{sym.}) &  & \mathbb{E}\left[\tilde{w}_{S}(\mathbf{X})\hat{w}_{S}(\mathbf{X})-\mathbb{E}\left\{ \tilde{w}_{S}(\mathbf{X})\hat{w}_{S}(\mathbf{X})\right\} \right]^{2}
\end{array}\!\!\right]\left\{ \!\!\!\begin{array}{c}
\alpha_{S,n}\\
\beta_{S,n}\\
\gamma_{S,n}
\end{array}\!\!\!\right\} \\
\\
=\left\{ \begin{array}{c}
\tilde{\sigma}_{S}^{2}\\
\mathbb{E}\left[w(\mathbf{X})\hat{w}_{S}(\mathbf{X})\right]\\
\mathbb{E}\left[w(\mathbf{X})\tilde{w}_{S}(\mathbf{X})\hat{w}_{S}(\mathbf{X})\right]
\end{array}\right\}.
\end{array}\label{36b}
\end{equation}
\end{theorem}

\begin{proof}
Define
\begin{equation}
\bar{e}_{S,n}:= \left[ w(\mathbf{X})-\bar{w}_{S,n}(\mathbf{X};\alpha_{S,n},\beta_{S,n},\gamma_{S,n})
\right]^2
\nonumber
\end{equation}
as the square of the difference between $w(\mathbf{X})$ and its $S$-variate hybrid approximation $\bar{w}_{S,n}(\mathbf{X};\alpha_{S,n},\beta_{S,n},\gamma_{S,n})$. For the mean-squared error of $\bar{w}_{S,n}$ to be minimum, set
\begin{equation}
\dfrac{\partial \mathbb{E} \left[ \bar{e}_{S,n} \right]}{\partial \alpha_{S,n}} =
\dfrac{\partial \mathbb{E} \left[ \bar{e}_{S,n} \right]}{\partial \beta_{S,n}} =
\dfrac{\partial \mathbb{E} \left[ \bar{e}_{S,n} \right]}{\partial \gamma_{S,n}} = 0.
\label{37}
\end{equation}

Let $\cal{A}$ be an open subset of $\mathbb{R}^3$. Suppose that $\bar{e}_{S,n}:\mathbb{R}^N\times \cal{A} \to \mathbb{R}$ satisfies the following regularity conditions: (1) $\bar{e}_{S,n}$ is a Lebesgue-integrable function of $\mathbf{X}$ for each $(\alpha_{S,n},\beta_{S,n},\gamma_{S,n})\in \cal{A}$; (2) for almost all $\mathbf{x} \in \mathbb{R}^N$, the partial derivatives ${\partial \bar{e}_{S,n}}/{\partial \alpha_{S,n}}$, ${\partial \bar{e}_{S,n}}/{\partial \beta_{S,n}}$, and ${\partial \bar{e}_{S,n}}/{\partial \gamma_{S,n}}$ exist for all $(\alpha_{S,n},\beta_{S,n},\gamma_{S,n})\in \cal{A}$; and (3) there exists an integrable function $\theta:\mathbb{R}^N \to \mathbb{R}$ such that
$|{\partial \bar{e}_{S,n}}/{\partial \alpha_{S,n}}|\le \theta(\mathbf{x})$,
$|{\partial \bar{e}_{S,n}}/{\partial \beta_{S,n}}|\le \theta(\mathbf{x})$, and
$|{\partial \bar{e}_{S,n}}/{\partial \gamma_{S,n}}|\le \theta(\mathbf{x})$ for all $(\alpha_{S,n},\beta_{S,n},\gamma_{S,n})\in \cal{A}$.  Then for all $(\alpha_{S,n},\beta_{S,n},\gamma_{S,n})\in \cal{A}$, the differential and expectation (integral) operators in (\ref{37}) can be interchanged, that is,
\begin{equation}
\mathbb{E} \left[ \dfrac{\partial \bar{e}_{S,n}}{\partial \alpha_{S,n}} \right] =
\mathbb{E} \left[ \dfrac{\partial \bar{e}_{S,n}}{\partial \beta_{S,n}} \right] =
\mathbb{E} \left[ \dfrac{\partial \bar{e}_{S,n}}{\partial \gamma_{S,n}} \right] = 0.
\label{38}
\end{equation}
Substituting the expression of $\bar{w}_{S,n}(\mathbf{X};\alpha_{S,n},\beta_{S,n},\gamma_{S,n})$
from (\ref{36}) into (\ref{38}) produces (\ref{36b}), proving the
theorem.~\end{proof}

\begin{corollary}
\label{cor:3}Neglecting the nonlinear term in (\ref{36}) through
(\ref{38}) creates an optimal, linear, $S$-variate hybrid approximation
\begin{equation}
\bar{w}_{S,l}(\mathbf{X};\alpha_{S,l},\beta_{S,l})=\alpha_{S,l}\tilde{w}_{S}(\mathbf{X})+\beta_{S,l}\hat{w}_{S}(\mathbf{X})\label{39}
\end{equation}
of $w(\mathbf{X})$, $1\le S<N<\infty$, where the optimal model parameters
\begin{equation}
\alpha_{S,l}=\dfrac{\hat{\sigma}_{S}^{2}\tilde{\sigma}_{S}^{2}-\mathbb{E}\left[\tilde{w}_{S}(\mathbf{X})\hat{w}_{S}(\mathbf{X})\right]\mathbb{E}\left[w(\mathbf{X})\hat{w}_{S}(\mathbf{X})\right]}{\tilde{\sigma}_{S}^{2}\hat{\sigma}_{S}^{2}-\left(\mathbb{E}\left[\tilde{w}_{S}(\mathbf{X})\hat{w}_{S}(\mathbf{X})\right]\right)^{2}}\label{40}
\end{equation}
and
\begin{equation}
\beta_{S,l}=\dfrac{\tilde{\sigma}_{S}^{2}\mathbb{E}\left[w(\mathbf{X})\hat{w}_{S}(\mathbf{X})\right]-\tilde{\sigma}_{S}^{2}\mathbb{E}\left[\tilde{w}_{S}(\mathbf{X})\hat{w}_{S}(\mathbf{X})\right]}{\tilde{\sigma}_{S}^{2}\hat{\sigma}_{S}^{2}-\left(\mathbb{E}\left[\tilde{w}_{S}(\mathbf{X})\hat{w}_{S}(\mathbf{X})\right]\right)^{2}}.\label{41}
\end{equation}
\end{corollary}

\begin{corollary}
\label{cor:4}Neglecting the nonlinear term and constraining the sum
of two remaining model parameters to be unity in (\ref{36}) through
(\ref{38}) creates yet another optimal, linear, $S$-variate hybrid
approximation
\begin{equation}
\bar{w}_{S,l'}(\mathbf{X};\alpha_{S,l'})=\alpha_{S,l'}\tilde{w}_{S}(\mathbf{X})+(1-\alpha_{S,l'})\hat{w}_{S}(\mathbf{X})\label{42}
\end{equation}
of $w(\mathbf{X})$, $1\le S<N<\infty$, where the optimal model parameter
\begin{equation}
\alpha_{S,l'}=\dfrac{\tilde{\sigma}_{S}^{2}+\hat{\sigma}_{S}^{2}-
\mathbb{E}\left[ \tilde{w}_{S}(\mathbf{X})\hat{w}_{S}(\mathbf{X})\right]-
\mathbb{E}\left[ w(\mathbf{X})\hat{w}_{S}(\mathbf{X})\right]}
{\tilde{\sigma}_{S}^{2}+\hat{\sigma}_{S}^{2}-
2\mathbb{E}\left[ \tilde{w}_{S}(\mathbf{X})\hat{w}_{S}(\mathbf{X})\right]}.
\label{43}
\end{equation}

\end{corollary}

The second linear hybrid approximation $\bar{w}_{S,l'}$ for $S=1$ or 2 presented in Corollary \ref{cor:4} coincides with that proposed by Tunga and Demiralp \cite{tunga06}.  However, the results of first linear hybrid approximation $\bar{w}_{S,l}$ and nonlinear hybrid approximation $\bar{w}_{S,n}$ $-$ that is, Theorem \ref{thm:2} and Corollary \ref{cor:3} $-$ are new.  Furthermore, the two linear approximations, $\bar{w}_{S,l}$ and $\bar{w}_{S,l'}$, are not the same for a general truncation $2\le S<N$.  Indeed,
\begin{equation}
\begin{array}{ll}
\alpha_{S,l}-\alpha_{S,l'} = & \\
\!\!\!\dfrac{
\left(\hat{\sigma}_{S}^{2}-\mathbb{E}\left[ \tilde{w}_{S}(\mathbf{X})\hat{w}_{S}(\mathbf{X})\right]\right)
\left(\tilde{\sigma}_{S}^{2}-\mathbb{E}\left[\tilde{w}_{S}(\mathbf{X})\hat{w}_{S}(\mathbf{X})\right]\right)
\left(\mathbb{E}\left[\tilde{w}_{S}(\mathbf{X})\hat{w}_{S}(\mathbf{X})\right]-\mathbb{E}\left[w(\mathbf{X})\hat{w}_{S}(\mathbf{X})\right] \right)}
{
\left( \tilde{\sigma}_{S}^{2}+\hat{\sigma}_{S}^{2}-2\mathbb{E}\left[ \tilde{w}_{S}(\mathbf{X})\hat{w}_{S}(\mathbf{X})\right] \right)
\left( \tilde{\sigma}_{S}^{2}\hat{\sigma}_{S}^{2}-\left\{\mathbb{E}\left[\tilde{w}_{S}(\mathbf{X})\hat{w}_{S}(\mathbf{X})\right]\right\}^{2}\right)}
&
\end{array}
\nonumber
\end{equation}
does not vanish for arbitrary $S$.  Nor do the model parameters of the first linear hybrid approximation, which satisfy
\begin{equation}
\alpha_{S,l}+\beta_{S,l}= 1 -
\dfrac{
\left(\tilde{\sigma}_{S}^{2}-\mathbb{E}\left[\tilde{w}_{S}(\mathbf{X})\hat{w}_{S}(\mathbf{X})\right]\right)
\left(\mathbb{E}\left[\tilde{w}_{S}(\mathbf{X})\hat{w}_{S}(\mathbf{X})\right]-\mathbb{E}\left[w(\mathbf{X})\hat{w}_{S}(\mathbf{X})\right] \right)}
{
\tilde{\sigma}_{S}^{2}\hat{\sigma}_{S}^{2}-\left\{\mathbb{E}\left[\tilde{w}_{S}(\mathbf{X})\hat{w}_{S}(\mathbf{X})\right]\right\}^{2}},
\nonumber
\end{equation}
sum to unity, as required in the second linear hybrid approximation.  Equation \ref{39} is endowed with two independent parameters and hence more flexibility in forming an optimal approximation. Therefore, the first approximation is expected to produce more accurate results than the second approximation, to be illustrated in Example 2.

{\em Remark 4.} The nonlinear hybrid approximation $\bar{w}_{S,n}$ presented in Theorem \ref{thm:2} is motivated on furnishing more precise stochastic solutions than those by either version of the linear hybrid approximation. This is possible when all higher-order moments involved in (\ref{36b}) exist and the system matrix in (\ref{36b}) is invertible.  The need for nonlinear approximation become significant when only a univariate truncation is feasible, but not necessarily yielding sufficiently accurate statistical solutions by linear approximations.  The numerical results from linear and nonlinear approximations will be contrasted in Section 5.

{\em Remark 5.} Since $\mathbb{E}[\tilde{w}_{S}(\mathbf{X})\hat{w}_{S}(\mathbf{X})] \le \tilde{\sigma}_{S}^{2}\hat{\sigma}_{S}^{2}$ (Cauchy-Schwarz inequality), the sub-matrix formed by the first two rows and columns of the system matrix in (\ref{36b}) is invertible.  Therefore, the linear hybrid approximations exist for any function $y$ or $w$ with a finite variance.  Under appropriate conditions, the linear approximations,
$\bar{w}_{S,l}$ and $\bar{w}_{S,l'}$, for a given $1\le S<N<\infty$,
are capable of reproducing the exact solution: (1) If $w=\tilde{w}_{S}$,
then (\ref{40}), (\ref{41}), and (\ref{43}) yield $\alpha_{S,l}=\alpha_{S,l'}=1$
and $\beta_{S,l}=0$, and hence $\bar{w}_{S,l}=\bar{w}_{S,l'}=\tilde{w}_{S}=w$;
(2) If $w=\hat{w}_{S}$, then (\ref{40}), (\ref{41}), and (\ref{43})
produce $\alpha_{S,l}=\alpha_{S,l'}=0$ and $\beta_{S,l}=1$, and
therefore $\bar{w}_{S,l}=\bar{w}_{S,l'}=\hat{w}_{S}=w$. In Example
1, for instance, both functions $y_1$ and $y_2$ and their respective linear, univariate
HDD approximations are coincident, leading to exact variances by the
HDD approximation.
%\end{remark}

\subsection{Second-Moment Properties}

Applying the expectation operator on (\ref{36}), (\ref{39}), and (\ref{42}) yields the exact mean
\begin{equation}
\mathbb{E}\left[\bar{y}_{S,n}(\mathbf{X};\alpha_{S,n},\beta_{S,n},\gamma_{S,n})\right]=\mathbb{E}\left[\bar{y}_{S,l}(\mathbf{X};\alpha_{S,l},\beta_{S,l})\right]=\mathbb{E}\left[\bar{y}_{S,l'}(\mathbf{X};\alpha_{S,l'})\right]=y_{\emptyset},\label{44}
\end{equation}
 by all three hybrid approximations. However, their respective variances,
\begin{equation}
\begin{array}{rcl}
\bar{\sigma}_{S,n}^{2} & := & \mathbb{E}\left[\bar{w}_{S,n}^{2}(\mathbf{X};\alpha_{S,n},\beta_{S,n},\gamma_{S,n})\right]=\alpha_{S,n}^{2}\tilde{\sigma}_{S}^{2}+\beta_{S,n}^{2}\hat{\sigma}_{S}^{2}\\
 &  & +\gamma_{S,n}^{2}\mathbb{E}\left[\tilde{w}_{S}(\mathbf{X})\hat{w}_{S}(\mathbf{X})-\mathbb{E}\left\{ \tilde{w}_{S}(\mathbf{X})\hat{w}_{S}(\mathbf{X})\right\} \right]^{2}+2\alpha_{S,n}\beta_{S,n}\mathbb{E}\left[\tilde{w}_{S}(\mathbf{X})\hat{w}_{S}(\mathbf{X})\right]\\
 &  & +2\alpha_{S,n}\gamma_{S,n}\mathbb{E}\left[\tilde{w}_{S}^{2}(\mathbf{X})\hat{w}_{S}(\mathbf{X})\right]+2\beta_{S,n}\gamma_{S,n}\mathbb{E}\left[\tilde{w}_{S}(\mathbf{X})\hat{w}_{S}^{2}(\mathbf{X})\right],
\end{array}\label{45}
\end{equation}
\begin{equation}
\bar{\sigma}_{S,l}^{2} := \mathbb{E}\left[\bar{w}_{S,l}^{2}(\mathbf{X};\alpha_{S,l},\beta_{S,l})\right]
= \alpha_{S,l}^{2}\tilde{\sigma}_{S}^{2}+\beta_{S,l}^{2}\hat{\sigma}_{S}^{2}+2\alpha_{S,l}\beta_{S,l}\mathbb{E}\left[\tilde{w}_{S}(\mathbf{X})\hat{w}_{S}(\mathbf{X})\right],
\label{46}
\end{equation}
and
\begin{equation}
\begin{array}{rcl}
\bar{\sigma}_{S,l'}^{2} & := & \mathbb{E}\left[\bar{w}_{S,l'}^{2}(\mathbf{X};\alpha_{S,l'},\beta_{S,l'})\right]\\
 & = & \alpha_{S,l'}^{2}\tilde{\sigma}_{S}^{2}+(1-\alpha_{S,l'})^{2}\hat{\sigma}_{S}^{2}+2\alpha_{S,l'}(1-\alpha_{S,l'})\mathbb{E}\left[\tilde{w}_{S}(\mathbf{X})\hat{w}_{S}(\mathbf{X})\right],
\end{array}\label{47}
\end{equation}
while approximate, emerge steadily more accurate
as $S\to N$. Compared with ADD and FDD approximations, the HDD approximations proposed
require expectations of various products of $\tilde{w}_{S}(\mathbf{X})$
and $\hat{w}_{S}(\mathbf{X})$ to calculate their variances. They
simplify greatly for the univariate HDD approximation, to be discussed
in the following section.

\subsection{Example 2}

Consider the blended function
\begin{equation}
y = y_\emptyset + { \displaystyle \frac{
{\displaystyle \sum_{i=1}^{N}  X_{i}} -
\mathbb{E}\left[ {\displaystyle \sum_{i=1}^{N} X_{i}} \right]
                                      }
{
\sqrt{\mathbb{E}\left( {\displaystyle \sum_{i=1}^{N}  X_{i}} -
\mathbb{E}\left[ {\displaystyle \sum_{i=1}^{N} X_{i}} \right] \right)^2}
}
} +
{ \displaystyle \frac{
{\displaystyle \prod_{i=1}^{N}  X_{i}} -
\mathbb{E}\left[ {\displaystyle \prod_{i=1}^{N} X_{i}} \right]
                                      }
{
\sqrt{\mathbb{E}\left( {\displaystyle \prod_{i=1}^{N}  X_{i}} -
\mathbb{E}\left[ {\displaystyle \prod_{i=1}^{N} X_{i}} \right] \right)^2}
}
}
\label{48}
\end{equation}
of $N$ random variables, $X_{i},$ $i=1,\cdots,N$, which are independent,
identical, and uniformly distributed over {[}0,1{]}. From elementary
calculations, the mean and variance of $y(\mathbf{X})$ are $y_\emptyset$ and
$2+2\sqrt{N \times 3^{N-1}/(2^{2N}-3^N)}$, respectively. Since $y$ in (\ref{48}) follows the general structure of (\ref{32b}), all ADD component functions
{[}(\ref{32c}),{\ref{32d}}{]} and the resultant second-moment statistics {[}(\ref{32c}),(\ref{32e}){]} of an $S$-variate ADD approximation were obtained
exactly. However, the mean and variance of an $S$-variate FDD approximation
were estimated by a fully symmetric multidimensional integration rule
with nine generators \cite{genz83}. The variance of an $S$-variate
HDD approximation was obtained using (\ref{46}), where the expectation in the last term and those involved in the model
parameters {[}(\ref{40}), (\ref{41}){]} were calculated by the same integration rule. The objective of this example is to evaluate the accuracy of various ADD, FDD, and
HDD (both linear models) approximations in calculating several probabilistic
characteristics of $y$ for $N=5$ and $y_\emptyset=5$.

The means of FDD approximations for $S=1$, 2, 3, and 4, obtained using the aforementioned multidimensional rule and retaining five digits after the decimal point, are 5.00000, 5.00014, 5.00021, and 5.00002. They are practically equal to 5 -- the value of $y_{\emptyset}$ -- demonstrating that all four FDD approximations also produce the exact mean, at least in this example. However, a formal proof of this conjecture, applicable to FDD approximations of arbitrary functions, remains elusive.

Table \ref{table3} presents the errors in variances of $y$ calculated by univariate
($S=1$) to quadrivariate ($S=4$) truncations of ADD, FDD, and HDD.
The definition of the error is the same as in Example 1. Both ADD and FDD
commit smaller errors when the truncation $S$ increases, as expected.
Since the purely additive and purely multiplicative pieces of $y$
have been standardized, that is, they have zero means and unit variances,
the ADD and FDD errors have similar orders and trends. In contrast,
the HDD errors from the first linear approximation, given $1\le S \le 4$, are consistently lower than the corresponding ADD or FDD errors by almost an order of magnitude. However, the same observation does not hold true for the second linear hybrid approximation, which provides more precise results than ADD or FDD only for $S=1$ or 2, but the accuracy does not improve or even degrades for $S=3$ and $4$.  The superior accuracy of the first linear approximation is attributed to two independent model parameters, compared with a single model parameter of the second linear approximation.  Nonetheless, the hybrid approximation proposed is more accurate than ADD or FDD approximation, at least, for second-moment analysis.

\begin{table}
\caption{ADD, FDD, and HDD approximation errors in variances of $y$ in Example 2 $^{(\text{a})}$ }
\begin{centering}
\begin{tabular}{ccccc}
\hline
 & \multicolumn{4}{c}{{\footnotesize Approximation}}\tabularnewline
\cline{2-5}
{\footnotesize $S$} & {\footnotesize ADD} & {\footnotesize FDD} & {\footnotesize HDD (1st linear model)} & {\footnotesize HDD (2nd linear model)} \tabularnewline
\hline
{\footnotesize 1} & {\footnotesize 0.139942} & {\footnotesize $9.8251\times10^{-2}$} & {\footnotesize $2.2734\times10^{-2}$} & {\footnotesize $2.2734\times10^{-2}$} \tabularnewline
{\footnotesize 2} & {\footnotesize $3.9452\times10^{-2}$} & {\footnotesize $2.8834\times10^{-2}$} & {\footnotesize $1.5503\times10^{-3}$} & {\footnotesize $5.5528\times10^{-3}$}\tabularnewline
{\footnotesize 3} & {\footnotesize $5.9550\times10^{-3}$} & {\footnotesize $3.1717\times10^{-3}$} & {\footnotesize $1.8656\times10^{-4}$} & {\footnotesize $5.7923\times10^{-3}$}\tabularnewline
{\footnotesize 4} & {\footnotesize $3.7219\times10^{-4}$} & {\footnotesize $3.7008\times10^{-4}$} & {\footnotesize $1.4099\times10^{-5}$}& {\footnotesize $5.0992\times10^{-4}$} \tabularnewline
\hline
\end{tabular}
\par\end{centering}
{\footnotesize ~~~~(a)
The error is defined as the absolute difference between exact
and approximate}{\footnotesize \par}
{\footnotesize ~~~~~~~~~variances, divided by the exact variance.}{\footnotesize \par}
\label{table3}
\end{table}

Does the improvement of HDD approximation extend to other probabilistic
characteristics of $y$? Figures \ref{figure1} depict the probability density functions (PDFs) of $y$ and its univariate to quadrivariate
ADD, FDD, and HDD (first linear model) approximations. The PDFs, essentially the normalized
histograms, were estimated from $10^{6}$ Monte Carlo simulations
of $y$ and their various approximations. At univariate truncation
($S=1)$, neither ADD nor FDD provide acceptable results when compared
with the PDF of $y$. In contrast, the PDF derived from the univariate
or bivariate HDD approximation matches with the PDF of $y$ better
than the corresponding ADD or FDD approximation. The discrepancy between
the PDFs of $y$ and its approximation, whether ADD, FDD, or HDD,
diminishes as $S$ increases. At quadrivariate truncation ($S=4)$,
all four PDFs are practically the same, even at the tail regions. Any distinction, especially between the results of $y$ and its FDD or HDD approximation, is hardly visible in the naked eye.  Judging from the rates at which the PDFs of approximations are converging to the PDF of $y$,
the HDD approximation may also be beneficial to calculating the probability
distribution of a stochastic response. This in large measure depends on the smoothness properties of the function chosen. However, strictly speaking, there does not exist a formal proof yet that an HDD approximation will always lead to accurate calculation of probabilities.  The topic merits further study.

\begin{figure}[h]
\begin{centering}
\includegraphics[scale=0.63]{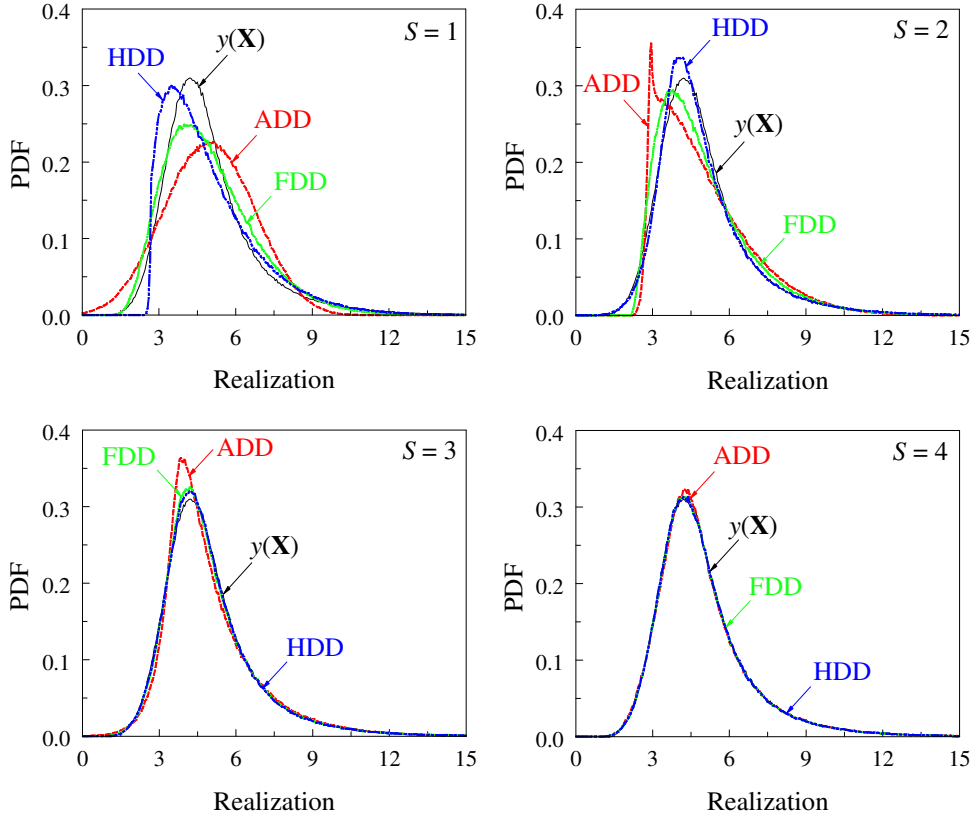}
\par\end{centering}
\caption{PDFs of $y$ and its ADD, FDD, and HDD approximations in Example 2; note: univariate (top, left); bivariate (top, right); trivariate (bottom, left); quadrivariate (bottom, right).}
\label{figure1}
\end{figure}

\subsection{Effective Dimensions}

When employing dimensional decomposition of a square-integrable function
$y$, an important decision revolves around selecting the truncation
parameter $S$. The truncation can be achieved by the notion of effective
dimension, introduced by Caflisch \emph{et al}. \cite{caflisch97},
who exploited ADD-based low effective dimension to explain why the
quasi Monte Carlo method outperforms the crude Monte Carlo algorithm for
evaluating a certain class of high-dimensional integrals. In this
subsection, a complete set of definitions, including two new alternative
ones, of effective dimension, stemming from ADD, FDD, and HDD are
presented.
\begin{definition}
\label{def:1}A square-integrable multivariate function $y$ with finite variance $0<\sigma^2<\infty$ has the ADD-based effective dimension $1\le S_{ADD}\le N$ if
\begin{equation}
S_{ADD}:=\min\left\{ S:1\le S\le N\;\text{such\; that\;}|\sigma^{2}-\tilde{\sigma}_{S}^{2}|\le(1-p)\sigma^{2}\right\} ,\label{49}
\end{equation}
the FDD-based effective dimension $1\le S_{FDD}\le N$ if
\begin{equation}
S_{FDD}:=\min\left\{ S:1\le S\le N\;\text{such\; that\;}|\sigma^{2}-\hat{\sigma}_{S}^{2}|\le(1-p)\sigma^{2}\right\} ,\label{50}
\end{equation}
and the HDD-based effective dimension $1\le S_{HDD}\le N$ if
\begin{equation}
S_{HDD}:=\min\left\{ S:1\le S\le N\;\text{such\; that\;}|\sigma^{2}-\bar{\sigma}_{S}^{2}|\le(1-p)\sigma^{2}\right\} ,\label{51}
\end{equation}
where $\bar{\sigma}_{S}^{2}:=\mathbb{E}[\bar{y}_{S}(\mathbf{X};\alpha_{S},\beta_{S},\gamma_{S},\cdots)-y_{\emptyset}]^{2}
= \mathbb{E}[\bar{w}_{S}^2]$ is the variance due to an $S$-variate (linear or nonlinear) HDD approximation of $w(\mathbf{X})$ and $0\le p\le1$ is a percentile threshold close
to \emph{one}.
\end{definition}

The definitions of effective dimension capture the notion in which $y$, represented by its ADD, FDD, or
HDD, is almost $S_{ADD}$-, $S_{FDD}$- or $S_{HDD}$-dimensional. The knowledge of effective dimension, whether rooted in ADD, FDD, or HDD, is valuable, because ANOVA terms up to the effective dimension contribute most to the function. Among these three definitions, the ADD approximation results in the
effective superposition dimension ($S_{ADD}$) of Caflisch \emph{et al}. $-$ a
concept ideally suited to a function comprising strongly additive
dimensional structure. In contrast, the effective dimension
$S_{FDD}$ proposed is appropriate when the dimensional
structure is strongly multiplicative. However, the dimensional hierarchy
of a multivariate function in general is neither dominantly additive
nor dominantly multiplicative.  Therefore, the effective dimension $S_{HDD}$
should be viewed as a generalized effective dimension that is applicable
to a broader class of functions than the ADD- and FDD-based effective
dimensions.  Furthermore, given a truncation $S$, an HDD approximation can be more precise than both ADD and FDD approximations.  In which case, the HDD-based effective dimension is expected to be lower than or equal to the ADD- or FDD-based effective dimensions.  It is, therefore, possible to optimally blend lower-variate ANOVA terms in HDD to contribute more to the function than by the same ANOVA terms in ADD or FDD alone.  See Example 3 for numerical results.

%\begin{remark}
{\em Remark 6.} Caflisch \emph{et al}. used the 99th percentile for $p$, but it can be treated as a threshold parameter linked to the desired accuracy of a stochastic solution. They also introduced effective dimension in the truncation sense, referred to as effective truncation dimension, which is appropriate when some variables are more important than others in an ordered set of the additive decomposition. It is possible to define similar effective dimensions grounded on FDD and HDD approximations as well.  In contrast, the effective superposition dimensions in (\ref{49}), (\ref{50}), and (\ref{51}) determine whether the low-variate component functions of dimensional decomposition dominate the function and are appropriate when all variables are equally important. For truly high-dimensional problems, all variables contribute to a function value; therefore, the effective superposition dimensions examined are more useful than the effective truncation dimension.
%\end{remark}

\subsection{Example 3}

Consider the functions $y_1$ and $y_2$ described by (\ref{34}) in Example 1 and $y$ described by (\ref{48}) in Example 2.  The purpose of this example is to compare the effective dimensions of these three functions, stemming from various dimensional decompositions and demonstrate the benefit of the HDD-based effective dimension.

Table \ref{table4} enumerates $S_{ADD}$, $S_{FDD}$, and $S_{HDD}$ of $y_1$ and $y_2$ for $N=6,7,8,9,10$ and of $y$ for $N=5$ when $p=0.99$.  They are easily obtained from respective definitions of effective dimensions given by (\ref{49}), (\ref{50}), and (\ref{51}), and the second-moment errors listed in Tables \ref{table1} and \ref{table2}.  For all functions or problem sizes examined, HDD leads to least effective dimensions. This is because the variance of a multivariate function, given the largest dimension of ANOVA component functions retained, is better estimated by HDD than ADD or FDD approximations.

\begin{table}
\caption{Effective dimensions of functions in Examples 1 and 2 for $p=0.99$}
\begin{centering}
\begin{tabular}{cccccccccccc}
\hline
 & \multicolumn{3}{c}{{\footnotesize Function $y_{1}$ in (\ref{34})}} &  & \multicolumn{3}{c}{{\footnotesize Function $y_{2}$ in (\ref{34})}} &  & \multicolumn{3}{c}{{\footnotesize Function $y$ in (\ref{48})}}\tabularnewline
\cline{2-4} \cline{6-8} \cline{10-12}
{\footnotesize $N$} & {\footnotesize $S_{ADD}$} & {\footnotesize $S_{FDD}$} & {\footnotesize $S_{HDD}$} &  & {\footnotesize $S_{ADD}$} & {\footnotesize $S_{FDD}$} & {\footnotesize $S_{HDD}$} &  & {\footnotesize $S_{ADD}$} & {\footnotesize $S_{FDD}$} & {\footnotesize $S_{HDD}$}\tabularnewline
\hline
{\footnotesize 5} & $-^{(\mathrm{a})}$ & $-^{(\mathrm{a})}$ & $-^{(\mathrm{a})}$ &  & $-^{(\mathrm{a})}$ & $-^{(\mathrm{a})}$ & $-^{(\mathrm{a})}$ &  & {\footnotesize 3} & {\footnotesize 3} & {\footnotesize 2}\tabularnewline
{\footnotesize 6} & {\footnotesize 4} & {\footnotesize 1} & {\footnotesize 1} &  & {\footnotesize 1} & {\footnotesize 2} & {\footnotesize 1} &  & $-^{(\mathrm{a})}$ & $-^{(\mathrm{a})}$ & $-^{(\mathrm{a})}$\tabularnewline
{\footnotesize 7} & {\footnotesize 5} & {\footnotesize 1} & {\footnotesize 1} &  & {\footnotesize 1} & {\footnotesize 2} & {\footnotesize 1} &  & $-^{(\mathrm{a})}$ & $-^{(\mathrm{a})}$ & $-^{(\mathrm{a})}$\tabularnewline
{\footnotesize 8} & {\footnotesize 5} & {\footnotesize 1} & {\footnotesize 1} &  & {\footnotesize 1} & {\footnotesize 2} & {\footnotesize 1} &  & $-^{(\mathrm{a})}$ & $-^{(\mathrm{a})}$ & $-^{(\mathrm{a})}$\tabularnewline
{\footnotesize 9} & {\footnotesize 6} & {\footnotesize 1} & {\footnotesize 1} &  & {\footnotesize 1} & {\footnotesize 2} & {\footnotesize 1} &  & $-^{(\mathrm{a})}$ & $-^{(\mathrm{a})}$ & $-^{(\mathrm{a})}$\tabularnewline
{\footnotesize 10} & {\footnotesize 6} & {\footnotesize 1} & {\footnotesize 1} &  & {\footnotesize 1} & {\footnotesize 2} & {\footnotesize 1} &  & $-^{(\mathrm{a})}$ & $-^{(\mathrm{a})}$ & $-^{(\mathrm{a})}$\tabularnewline
\hline
\end{tabular}
\par\end{centering}
{\footnotesize (a) Not applicable as $6 \le N \le 10$ in Example 1 and $N=5$ in Example 2.}\\
\label{table4}
\end{table}

\section{Univariate Approximations}

The univariate truncation ($S=1$) of each decomposition, whether
ADD, FDD, or HDD, constitutes simplest possible approximation of $y$
other than its mean. In which case, the approximate solutions of statistical
moments and error analysis simplify and can be obtained or conducted
in closed-form or analytically. They are of huge practical interest, because modeling and simulation of large-scale complex systems carry a heavy computational price tag, for which only a univariate truncation is feasible.

\subsection{ADD and FDD Error Analyses}

Consider the univariate ADD approximation
\begin{equation}
\tilde{y}_{1}(\mathbf{X})=y_{\emptyset}+{\displaystyle \sum_{i=1}^{N}y_{\{i\}}(X_{i})}\label{52}
\end{equation}
and the univariate FDD approximation
\begin{equation}
\hat{y}_{1}(\mathbf{X})=
(1+z_\emptyset){\displaystyle \prod_{i=1}^N \left[ 1 + z_{\{i\}}(X_i)\right]} =
{\displaystyle y_{\emptyset}}{\displaystyle {\displaystyle \prod_{i=1}^{N}}\frac{{\displaystyle {\displaystyle {\displaystyle y_{\emptyset}}+y_{\{i\}}(X_{i})}}}{{\displaystyle y_{\emptyset}}}}\label{53}
\end{equation}
of $y(\mathbf{X})$, obtained by fixing $S=1$ in (\ref{3}) and (\ref{16}),
respectively. Applying the expectation operator
on $\tilde{y}_{1}(\mathbf{X})$, $\hat{y}_{1}(\mathbf{X})$, and their
squares, the means
\begin{equation}
\mathbb{E}[\tilde{y}_{1}(\mathbf{X})]=\mathbb{E}[\hat{y}_{1}(\mathbf{X})]={\displaystyle \mathbb{E}[y(\mathbf{X})]:=y_{\emptyset}}\nonumber
\end{equation}
are identical and exact, although the respective variances,
\begin{equation}
\tilde{\sigma}_{1}^{2}:=\mathbb{E}\left[\left(\tilde{y}_{1}(\mathbf{X})-y_{\emptyset}\right)^{2}\right]=\sum_{i=1}^{N}\sigma_{\{i\}}^{2}\;\text{and}\;\label{54}
\end{equation}
\begin{equation}
\hat{\sigma}_{1}^{2}:=\mathbb{E}\left[\left(\hat{y}_{1}(\mathbf{X})-y_{\emptyset}\right)^{2}\right]={\displaystyle y_{\emptyset}^{2}}\left[{\displaystyle \prod_{i=1}^{N}}\left(1+\frac{{\displaystyle {\displaystyle \sigma_{\{i\}}^{2}}}}{{\displaystyle y_{\emptyset}^{2}}}\right)-1\right],\label{55}
\end{equation}
are approximate, where $\sigma_{\{i\}}^{2}:=\mathbb{E}\left[y_{\{i\}}^{2}(X_{i})\right]$
is the variance of the \emph{zero}-mean, univariate ADD component
function $y_{\{i\}}$, $i=1,\cdots,N$. The last expression of (\ref{55})
is obtained from the realization that $\mathbb{E}[{\displaystyle {\displaystyle \{{\displaystyle y_{\emptyset}}+y_{\{i\}}(X_{i})\}/{\displaystyle y_{\emptyset}}}}]=1$
and $\mathbb{E}[{\displaystyle {\displaystyle \{{\displaystyle y_{\emptyset}}+y_{\{i\}}(X_{i})\}^{2}/{\displaystyle y_{\emptyset}^{2}}}}]=1+\sigma_{\{i\}}^{2}/{\displaystyle y_{\emptyset}^{2}}$.
Comparing (\ref{54}) and (\ref{55}),  $\hat{\sigma}_{1}^{2}\to\tilde{\sigma}_{1}^{2}$ when $\sigma_{\{i\}}^{2}/{\displaystyle y_{\emptyset}^{2}\to0}$ or when $\sigma_{\{i\}}^{2}/{\displaystyle y_{\emptyset}^{2}<1}$ and $N\to\infty$. See the original work of Hoeffding \cite{hoeffding48} for the existence of ADD for $N \to \infty$.

Given a general square-integrable function $y$, are these univariate
approximations adequate? Which approximation between ADD and FDD is
more precise? The resolution of these questions depends on error analysis, where
\begin{equation}
\tilde{e}_{1}:=\mathbb{E}\left[y(\mathbf{X})-\tilde{y}_{1}(\mathbf{X})\right]^{2}\;\text{and}\;\hat{e}_{1}:=\mathbb{E}\left[y(\mathbf{X})-\hat{y}_{1}(\mathbf{X})\right]^{2}\label{56}
\end{equation}
define two second-moment errors, stemming from univariate ADD and
FDD approximations of $y$. Applying (\ref{1a}), (\ref{52}), (\ref{53})
into (\ref{56}) and invoking Propositions \ref{prop:1} and \ref{prop:2},
\begin{equation}
\tilde{e}_{1}=\sum_{s=2}^{N}\:\sum_{{\textstyle {\emptyset\ne u\subseteq\{1,\cdots,N\}\atop |u|=s}}}\sigma_{u}^{2}\label{58}
\end{equation}
is a sum of variance terms contributed by bivariate and higher-variate
component functions of $y$, whereas
\begin{equation}
\begin{array}{rcl}
\hat{e}_{1} & = & {\displaystyle \sum_{s=2}^{N}}~{\displaystyle \sum_{{\textstyle {\emptyset\ne u\subseteq\{1,\cdots,N\}\atop |u|=s}}}}\sigma_{u}^{2}+
y_{\emptyset}^{2}\:\mathbb{E} \left[
{\displaystyle \sum_{s=2}^{N}}~{\displaystyle \sum_{{\textstyle {\emptyset\ne u\subseteq\{1,\cdots,N\}\atop |u|=s}}}}~
{\displaystyle \prod_{i\in u} \frac{y_{\{i\}}(X_{i})}{y_{\emptyset}}}
\right]^2 - \\
 &  & 2y_{\emptyset}{\displaystyle \sum_{s=2}^{N}\:}
\mathbb{E} \left[
\left\{
{\displaystyle \sum_{s=2}^{N}}~{\displaystyle \sum_{{\textstyle {\emptyset\ne u\subseteq\{1,\cdots,N\}\atop |u|=s}}}}~
{\displaystyle \prod_{i\in u} \frac{y_{\{i\}}(X_{i})}{y_{\emptyset}}}
\right\}
\left\{
{\displaystyle \sum_{s=2}^{N}}~{\displaystyle \sum_{{\textstyle {\emptyset\ne u\subseteq\{1,\cdots,N\}\atop |u|=s}}}}\:y_u(\mathbf{X}_u)
\right\}
\right]
\end{array}\label{59}
\end{equation}
includes two additional terms, comprising expectations of products
of three or more component functions of $y$. Since the first term
of (\ref{59}) is the same as $\tilde{e}_{1}$, the univariate FDD
approximation incurs less (more) error than the univariate ADD approximation
if the second term is smaller (larger) than the third term. Therefore,
the condition
\begin{equation}
\begin{array}{c}
\mathbb{E} \left[
{\displaystyle \sum_{s=2}^{N}}~{\displaystyle \sum_{{\textstyle {\emptyset\ne u\subseteq\{1,\cdots,N\}\atop |u|=s}}}}~
{\displaystyle \prod_{i\in u} \frac{y_{\{i\}}(X_{i})}{y_{\emptyset}}}
\right]^2
< \\
{\displaystyle \frac{2}{y_{\emptyset}} }
{\displaystyle \sum_{s=2}^{N}\:}
\mathbb{E} \left[
\left\{
{\displaystyle \sum_{s=2}^{N}}~{\displaystyle \sum_{{\textstyle {\emptyset\ne u\subseteq\{1,\cdots,N\}\atop |u|=s}}}}~
{\displaystyle \prod_{i\in u} \frac{y_{\{i\}}(X_{i})}{y_{\emptyset}}}
\right\}
\left\{
{\displaystyle \sum_{s=2}^{N}}~{\displaystyle \sum_{{\textstyle {\emptyset\ne u\subseteq\{1,\cdots,N\}\atop |u|=s}}}}\:y_u(\mathbf{X}_u)
\right\}
\right]
\end{array}
\nonumber
\end{equation}
determines when the univariate FDD approximation is more precise than
the univariate ADD approximation and vice versa. When $N=2$, for
instance, $\hat{e}_{1}<\tilde{e}_{1}$ if $\mathbb{E}[y_{\{1\}}^{2}(X_{1})y_{\{2\}}^{2}(X_{2})]< 2y_{\emptyset}
\mathbb{E}[y_{\{1\}}(X_{1})y_{\{2\}}(X_{2})y_{\{12\}}(X_{1},X_{2})]$
and vice versa.

%\begin{remark}
{\em Remark 7.} The ADD approximation $\tilde{y}_{1}(\mathbf{X})$ is called
univariate because (\ref{52}) comprises a sum of at most univariate
component functions, describing only the main effect of $\mathbf{X}$.
In contrast, the FDD approximation $\hat{y}_{1}(\mathbf{X})$ in (\ref{53})
contains products of various univariate functions. Therefore, some
effects of interactions between two input variables $X_{i}$ and $X_{j}$,
$i\ne j$, subsist in $\hat{y}_{1}(\mathbf{X})$. For example, consider
a function $y=y_{\emptyset}+{\displaystyle y_{\{1\}}(X_{1})}+{\displaystyle y_{\{2\}}(X_{2})}+y_{\{1\}}(X_{1})y_{\{2\}}(X_{2})/y_{\emptyset}$
of two variables, containing a sum and a product of its univariate
ANOVA component functions. The univariate ADD approximation, $\tilde{y}_{1,m}=y_{\emptyset}+{\displaystyle y_{\{1\}}(X_{1})}+{\displaystyle y_{\{2\}}(X_{2})}$,
captures only the main effects of $X_{1}$ and $X_{2}$, and may produce
non-negligible errors if the product term of $y$ is significant.
On the other hand, the univariate FDD approximation, $\hat{y}_{1}=(1+z_{\emptyset})[1+z_{\{1\}}(X_{1})][1+z_{\{2\}}(X_{2})]=y_{\emptyset}+{\displaystyle y_{\{1\}}(X_{1})}+{\displaystyle y_{\{2\}}(X_{2})}+y_{\{1\}}(X_{1})y_{\{2\}}(X_{2})/y_{\emptyset}$,
obtained using the relationships in (\ref{13}) and (\ref{14}), exactly
reproduces $y$, thereby capturing not only the main effects, but
also the interactive effect of input variables. Similar conditions
prevail for a function of an arbitrary, but finite, number of variables, provided that the higher-variate ANOVA component functions are products of univariate ANOVA component functions.  However, the ANOVA component functions with distinct dimensions are unrelated in general.  Therefore, the error committed by an FDD approximation may or may not be lower than that by an ADD approximation.  Nonetheless, the term ``univariate'' used in this paper for the FDD approximation should be interpreted in the context of including at most univariate component functions, not necessarily preserving only the main effects.
%\end{remark}

There exist two special cases where one of the two univariate
approximations does not perpetrate any error. First, consider a purely
additive function $y$, where its zero-variate and univariate component
functions are arbitrary, but its bivariate and higher-variate ANOVA
component functions vanish, that is, $y_{u}(\mathbf{X}_{u})=0$ for
$2\le|u|\le N$. In this case, (\ref{58}) and (\ref{59}) yield $\tilde{e}_{1}=0$,
while
\begin{equation}
\hat{e}_{1}=
y_{\emptyset}^{2}\:\mathbb{E} \left[
{\displaystyle \sum_{s=2}^{N}}~{\displaystyle \sum_{{\textstyle {\emptyset\ne u\subseteq\{1,\cdots,N\}\atop |u|=s}}}}~
{\displaystyle \prod_{i\in u} \frac{y_{\{i\}}(X_{i})}{y_{\emptyset}}}
\right]^2
>0,\nonumber
\end{equation}
indicating superiority of univariate ADD over univariate FDD approximations.
Second, consider a function $y$, where bivariate and higher-variate
ANOVA component functions are products of zero-variate and univariate
component functions, distributed as follows: $y_{u}(\mathbf{X}_{u})={\displaystyle y_{\emptyset}}\prod_{i\in u\subseteq\{1,\cdots,N\}}\: y_{\{i\}}(X_{i})/y_{\emptyset}]$
for $2\le|u|\le N$. In the latter case, (\ref{58}) and (\ref{59})
deliver
\begin{equation}
\tilde{e}_{1}= y_{\emptyset}^{2}\:
{\displaystyle \sum_{s=2}^{N}}~{\displaystyle \sum_{{\textstyle {\emptyset\ne u\subseteq\{1,\cdots,N\}\atop |u|=s}}}}~
\mathbb{E} \left[
{\displaystyle \prod_{i\in u} \frac{y_{\{i\}}^2(X_{i})}{y_{\emptyset}^2}}
\right]
>0,\nonumber
\end{equation}
but $\hat{e}_{1}=0$, reversing the trend of the previous case.

\subsection{HDD}

The $S$-variate hybrid approximations and their statistics proposed
in the preceding section also simplify for univariate truncations
of ADD and FDD. Both linear and nonlinear hybrid approximations are
described by the following two propositions.
\begin{proposition}
\label{prop:3}A linear, univariate HDD approximation of $w(\mathbf{X})$,
obtained by setting $S=1$ in (\ref{39}) through (\ref{41}), is
\begin{equation}
\bar{w}_{1,l}(\mathbf{X};\alpha_{1,l},\beta_{1,l})=\alpha_{S,l}\tilde{w}_{1}(\mathbf{X})+\beta_{1,l}\hat{w}_{1}(\mathbf{X}),\label{60}
\end{equation}
where the model parameters
\begin{equation}
\alpha_{1,l}=\dfrac{\hat{\sigma}_{1}^{2}-\mathbb{E}\left[w(\mathbf{X})\hat{w}_{1}(\mathbf{X})\right]}{\hat{\sigma}_{1}^{2}-\tilde{\sigma}_{1}^{2}}\;\text{and}\;\beta_{1,l}=\dfrac{\mathbb{E}\left[w(\mathbf{X})\hat{w}_{1}(\mathbf{X})\right]-\tilde{\sigma}_{1}^{2}}{\hat{\sigma}_{1}^{2}-\tilde{\sigma}_{1}^{2}}.\label{61}
\end{equation}
\end{proposition}

The result in Proposition \ref{prop:3} is obtained using the relationships,
$\mathbb{E}[w(\mathbf{X})\tilde{w}_{1}(\mathbf{X})]=\mathbb{E}[\tilde{w}_{1}(\mathbf{X})\hat{w}_{1}(\mathbf{X})]=\mathbb{E}[\tilde{w}_{1}^{2}(\mathbf{X})]=:\tilde{\sigma}_{1}^{2}$,
that stem from Propositions \ref{prop:1} and \ref{prop:2}. The remaining
expectation
\begin{equation}
\mathbb{E}\left[w(\mathbf{X})\hat{w}_{1}(\mathbf{X})\right]=
\mathbb{E}\left[y(\mathbf{X})\hat{y}_{1}(\mathbf{X})\right]-y_{\emptyset}^{2}
\label{59b}
\end{equation}
cannot be reduced further as it involves the original function $y$
or $w$, but it can be estimated by sampling methods or numerical
integration.

%\begin{remark}
{\em Remark 8.} The two parameters of the first linear model, described by
(\ref{60}) and (\ref{61}), add up to \emph{one}. This is due to
special properties of $\tilde{y}_{1}(\mathbf{X})$ and $\hat{y}_{1}(\mathbf{X})$
discussed in Subsection 5.1. Therefore, the second linear model, described
by (\ref{42}) and (\ref{43}) at univariate truncation ($S=1$),
is redundant, as it leads to the same solution of the first linear
model. However, due to complicated forms of $\hat{y}_{S}(\mathbf{X})$, where
$2\le S<N$, the same relationship does not hold for two generic, linear, $S$-variate HDD approximations.  See Corollaries \ref{cor:3} and \ref{cor:4} for further insights.
%\end{remark}

The mean of $\bar{w}_{1,l}(\mathbf{X};\alpha_{1,l},\beta_{1,l})$
is \emph{zero} and, therefore, $\mathbb{E}[\bar{y}_{1,l}(\mathbf{X};\alpha_{1,l},\beta_{1,l})]=y_{\emptyset}$,
matching the exact mean. The variance of $\bar{w}_{1,l}(\mathbf{X};\alpha_{1,l},\beta_{1,l})$
or $\bar{y}_{1,l}(\mathbf{X};\alpha_{1,l},\beta_{1,l})$ is
\begin{equation}
\begin{array}{rcl}
\bar{\sigma}_{1,l}^{2} & := & \mathbb{E}\left[\bar{w}_{1,l}^{2}
(\mathbf{X};\alpha_{1,l},\beta_{1,l})\right] \\
                       & =  &
\left(\alpha_{1,l}^{2}+2\alpha_{1,l}\beta_{1,l}\right)\tilde{\sigma}_{1}^{2}+
\beta_{1,l}^{2}\hat{\sigma}_{1}^{2} \\
                       & =  &
\left(2\alpha_{1,l}-\alpha_{1,l}^{2}\right)\tilde{\sigma}_{1}^{2}+
\left(1-\alpha_{1,l}\right)^{2}\hat{\sigma}_{1}^{2},
\end{array}
\label{63}
\end{equation}
a linear combination of variances from univariate ADD and FDD approximations.
\begin{proposition}
\label{prop:4}A nonlinear, univariate HDD approximation of $w(\mathbf{X})$,
obtained by setting $S=1$ in (\ref{36}) through (\ref{38}), is
\begin{equation}
\bar{w}_{1,n}(\mathbf{X};\alpha_{1,n},\beta_{1,n},\gamma_{1,n}) = \alpha_{1,n}\tilde{w}_{1}(\mathbf{X})+\beta_{1,n}\hat{w}_{1}(\mathbf{X})+\gamma_{1,n}\left[\tilde{w}_{1}(\mathbf{X})\hat{w}_{1}(\mathbf{X})-\tilde{\sigma}_{1}^{2}\right],\nonumber
\end{equation}
where the model parameters are the solution of
\begin{equation}
\left[\!\!\begin{array}{ccc}
\tilde{\sigma}_{1}^{2} & \tilde{\sigma}_{1}^{2} & \mathbb{E}\left[\tilde{w}_{1}^{2}(\mathbf{X})\hat{w}_{1}(\mathbf{X})\right]\\
 & \hat{\sigma}_{1}^{2} & \mathbb{E}\left[\tilde{w}_{1}(\mathbf{X})\hat{w}_{1}^{2}(\mathbf{X})\right]\\
(\text{sym.}) &  & \mathbb{E}\left[\tilde{w}_{1}^{2}(\mathbf{X})\hat{w}_{1}^{2}(\mathbf{X})\right]-\tilde{\sigma}_{1}^{4}
\end{array}\!\!\right]\left\{ \!\!\!\begin{array}{c}
\alpha_{1,n}\\
\beta_{1,n}\\
\gamma_{1,n}
\end{array}\!\!\!\right\} =\left\{ \begin{array}{c}
\tilde{\sigma}_{1}^{2}\\
\mathbb{E}\left[w(\mathbf{X})\hat{w}_{1}(\mathbf{X})\right]\\
\mathbb{E}\left[w(\mathbf{X})\tilde{w}_{1}(\mathbf{X})\hat{w}_{1}(\mathbf{X})\right]
\end{array}\right\} .\label{65}
\end{equation}

\end{proposition}

Starting from (\ref{52}) and (\ref{53}) and applying Propositions
\ref{prop:1} and \ref{prop:2}, the additional expectations involved
in (\ref{65}) are
\begin{equation}
\mathbb{E}\left[\tilde{w}_{1}^{2}(\mathbf{X})\hat{w}_{1}(\mathbf{X})\right] = \frac{{\displaystyle {\displaystyle 2}}}{{\displaystyle y_{\emptyset}}}\sum_{i_{1}=1}^{N-1}\sum_{i_{2}=i_{1}+1}^{N}\sigma_{\{i_{1}\}}^{2}\sigma_{\{i_{2}\}}^{2}+\sum_{i=1}^{N}\mathbb{E}\left[y_{\{i\}}^{3}(X_{i})\right],\nonumber
\end{equation}
\begin{equation}
\mathbb{E}\left[\tilde{w}_{1}(\mathbf{X})\hat{w}_{1}^{2}(\mathbf{X})\right] = {\displaystyle \left(\hat{\sigma}_{1}^{2}+y_{\emptyset}^{2}\right)}\sum_{i=1}^{N}\frac{{\displaystyle {\displaystyle 2}}y_{\emptyset}\sigma_{\{i\}}^{2}+\mathbb{E}\left[y_{\{i\}}^{3}(X_{i})\right]}{\sigma_{\{i\}}^{2}+y_{\emptyset}^{2}}-{\displaystyle {\displaystyle 2}}y_{\emptyset}\tilde{\sigma}_{1}^{2},\nonumber
\end{equation}
\begin{equation}
\begin{array}{rcl}
\mathbb{E}\left[\tilde{w}_{1}^{2}(\mathbf{X})\hat{w}_{1}^{2}(\mathbf{X})\right]
& \!\!=\!\!\ &
\!\!{\displaystyle \left(\hat{\sigma}_{1}^{2} + y_{\emptyset}^{2}\right)}
\left(
{\displaystyle
\sum_{i=1}^{N}
\frac{y_{\emptyset}^{2}\sigma_{\{i\}}^{2}+2y_{\emptyset}\mathbb{E}
\left[y_{\{i\}}^{3}(X_{i})\right]+\mathbb{E}\left[y_{\{i\}}^{4}
(X_{i})\right]}{\sigma_{\{i\}}^{2}+y_{\emptyset}^{2}}}\right.  \\
& \!\! + \!\!&
 \!\! \left.2
{\displaystyle
 \sum_{i_{1}=1}^{N-1}\!\sum_{i_{2}=i_{1}+1}^{N}\!\!\!
 \frac{
 (2y_{\emptyset}\sigma_{\{i_{1}\}}^{2}+\mathbb{E}[y_{\{i_{1}\}}^{3}(X_{i_{1}})])
 (2y_{\emptyset}\sigma_{\{i_{2}\}}^{2}+\mathbb{E}[y_{\{i_{2}\}}^{3}(X_{i_{2}})])
      }
      {
  (\sigma_{\{i_{1}\}}^{2}+y_{\emptyset}^{2})
  (\sigma_{\{i_{2}\}}^{2}+y_{\emptyset}^{2})
      }
}
 \right) \\
 & \!\!+\!\! &
 \!\!y_{\emptyset}\tilde{\sigma}_{1}^{2} -
 2y_{\emptyset}{\displaystyle \sum_{i=1}^{N} \mathbb{E}\left[y_{\{i\}}^{3}(X_{i})\right]}-
 4 {\displaystyle
 \sum_{i_{1}=1}^{N-1}\sum_{i_{2}=i_{1}+1}^{N}\sigma_{\{i_{1}\}}^{2}
 \sigma_{\{i_{2}\}}^{2}},
 \end{array}
 \nonumber
\end{equation}
and
\begin{equation}
\mathbb{E}\left[w(\mathbf{X})\tilde{w}_{1}(\mathbf{X})\hat{w}_{1}(\mathbf{X})\right]=\mathbb{E}\left[y(\mathbf{X})\tilde{y}_{1}(\mathbf{X})\hat{y}_{1}(\mathbf{X})\right]-y_{\emptyset}\mathbb{E}\left[y(\mathbf{X})\hat{y}_{1}(\mathbf{X})\right]-{\displaystyle {\displaystyle 2}}y_{\emptyset}\tilde{\sigma}_{1}^{2},\label{70}
\end{equation}
expressed in terms of $y_{\emptyset}$ and various moments of univariate ANOVA component
functions. Compared with the linear, hybrid model, however, they require
third- and fourth-order moments that must be furnished. Similar to
(\ref{59b}), (\ref{70}) also involves $w$ or $y$ and cannot be
reduced further.

The mean of $\bar{w}_{1,n}(\mathbf{X};\alpha_{1,n},\beta_{1,n},\gamma_{1,n})$
is \emph{zero} and, therefore, $\mathbb{E}[\bar{y}_{1,n}(\mathbf{X};\alpha_{1,n},\beta_{1,n},\gamma_{1,n})]=y_{\emptyset}$,
matching the exact mean as well. However, the variance of $\bar{w}_{1,n}(\mathbf{X};\alpha_{1,n},\beta_{1,n},\gamma_{1,n})$
or $\bar{y}_{1,n}(\mathbf{X};\alpha_{1,n},\beta_{1,n},\gamma_{1,n})$
is
\begin{equation}
\begin{array}{rcl}
\bar{\sigma}_{1,n}^{2} & := & \mathbb{E}\left[\bar{w}_{1,n}^{2}(\mathbf{X};\alpha_{1,n},\beta_{1,n},\gamma_{1,n})\right]\\
 & = & \alpha_{1,n}^{2}\tilde{\sigma}_{1}^{2}+\beta_{1,n}^{2}\hat{\sigma}_{1}^{2}+
 \gamma_{1,n}^{2}\left(\mathbb{E}\left[\tilde{w}_{1}^{2}(\mathbf{X})\hat{w}_{1}^{2}
 (\mathbf{X})\right]-\tilde{\sigma}_{1}^{4} \right)
 +2\alpha_{1,n}\beta_{1,n}\tilde{\sigma}_{1}^{2}\\
 &  & +2\alpha_{1,n}\gamma_{1,n}\mathbb{E}\left[\tilde{w}_{1}^{2}(\mathbf{X})\hat{w}_{1}(\mathbf{X})\right]+2\beta_{1,n}\gamma_{1,n}\mathbb{E}\left[\tilde{w}_{1}(\mathbf{X})\hat{w}_{1}^{2}(\mathbf{X})\right],
\end{array}\label{71}
\end{equation}
a nonlinear combination of variances and higher-order moments from
univariate ADD and FDD approximations. It is trivial to show that
(\ref{71}) reduces to (\ref{63}) if $\alpha_{1,n}=\alpha_{1,l}$,
$\beta_{1,n}=\beta_{1,l}$, and $\gamma_{1,n}=0$.

\subsection{HDD Error Analysis}

For the univariate truncation, which approximation stemming from ADD, FDD, and HDD is most accurate?  Lemma \ref{lem:1} and Theorem \ref{thm:3} demonstrate that the HDD approximation commits the lowest error.

\begin{lemma}
\label{lem:1}
The variance of the univariate FDD approximation is greater than or equal to the variance of the univariate ADD approximation, that is, $\hat{\sigma}_{1}^{2} \ge \tilde{\sigma}_{1}^{2}$.
\end{lemma}

\begin{proof}
From (\ref{55}),
\begin{equation}
\begin{array}{rcl}
\hat{\sigma}_{1}^{2}  & = & {\displaystyle y_{\emptyset}^{2}}\left[{\displaystyle \prod_{i=1}^{N}}\left(1+\frac{{\displaystyle {\displaystyle \sigma_{\{i\}}^{2}}}}{{\displaystyle y_{\emptyset}^{2}}}\right)-1\right] \\
 &  =  & {\displaystyle \sum_{i=1}^N \sigma_{\{i\}}^2} +
 y_{\emptyset}^{2}{\displaystyle \sum_{s=2}^{N}}~{\displaystyle \sum_{{\textstyle {\emptyset\ne u\subseteq\{1,\cdots,N\}\atop |u|=s}}}}~
{\displaystyle \prod_{i\in u} \frac{\sigma_{\{i\}}^2}{y_{\emptyset}^2}} \\
 & \ge & \tilde{\sigma}_{1}^{2},
\end{array}
\nonumber
\end{equation}
where the last line follows from (\ref{54}) and the recognition that second term of the second line is non-negative.
\end{proof}

\begin{theorem}
\label{thm:3}
Let $\tilde{e}_{1}:=\mathbb{E}\left[y(\mathbf{X})-\tilde{y}_{1}(\mathbf{X})\right]^{2}$,
$\hat{e}_{1}:=\mathbb{E}\left[y(\mathbf{X})-\hat{y}_{1}(\mathbf{X})\right]^{2}$, and
$\bar{e}_{1,l}:=\mathbb{E}\left[y(\mathbf{X})-\bar{y}_{1,l}(\mathbf{X};\alpha_{1,l},\beta_{1,l})\right]^{2}$
define the mean-squared errors committed by the univariate ADD, univariate FDD, and univariate HDD (linear) approximations, respectively, of a real-valued, square-integrable function $y$.  Then
\begin{equation}
\bar{e}_{1,l} \le \tilde{e}_{1}, ~\bar{e}_{1,l} \le \hat{e}_{1}.
\nonumber
\end{equation}
\end{theorem}

\begin{proof}
From (\ref{54}) and (\ref{58}),
\begin{equation}
\tilde{e}_{1} = \sigma^2 - {\displaystyle \sum_{i=1}^N \sigma_{\{i\}}^2} = \sigma^2 - \tilde{\sigma}_1^2.
\label{sr1}
\end{equation}
Since $y$, $\hat{y}_{1}$, and $\bar{y}_{1,l}$ have the same mean,
\begin{equation}
\begin{array}{rcl}
\hat{e}_{1}  & = & \mathbb{E}\left[w(\mathbf{X})-\hat{w}_{1}(\mathbf{X})\right]^{2} \\
             & = & \sigma^2 + \hat{\sigma}_1^2 - 2\mathbb{E}\left[w(\mathbf{X})\hat{w}_{1}(\mathbf{X})\right] \\
             & = & \sigma^2 - 2\alpha_{1,l}\tilde{\sigma}_1^2 - \left( 1-2\alpha_{1,l}\right) \hat{\sigma}_1^2,
\end{array}
\label{sr2}
\end{equation}
\begin{equation}
\begin{array}{rcl}
\bar{e}_{1,l} & = & \mathbb{E}\left[w(\mathbf{X})-\bar{w}_{1,l}(\mathbf{X};\alpha_{1,l},\beta_{1,l})\right]^{2} \\
              & = & \sigma^2 + \bar{\sigma}_{1,l}^2 - 2\mathbb{E}\left[w(\mathbf{X})\bar{w}_{1,l}(\mathbf{X};;\alpha_{1,l},\beta_{1,l})\right] \\
              & = & \sigma^2 + \left(2\alpha_{1,l}-\alpha_{1,l}^2\right)\tilde{\sigma}_1^2 + \left( 1-\alpha_{1,l}\right)^2 \hat{\sigma}_1^2 - \\
              &   & 2\mathbb{E}\left[w(\mathbf{X})\left\{ \alpha_{1,l}\tilde{w}_{1}(\mathbf{X})+
              (1-\alpha_{1,l})\hat{w}_{1}(\mathbf{X}) \right\} \right] \\
              & = & \sigma^2 - \left(2\alpha_{1,l}-\alpha_{1,l}^2\right)\tilde{\sigma}_1^2 - \left( 1-\alpha_{1,l}\right)^2 \hat{\sigma}_1^2.
\end{array}
\label{sr3}
\end{equation}
The third line of (\ref{sr2}) is obtained using (\ref{61}).  In (\ref{sr3}), the third line is derived using (\ref{60}), (\ref{61}), and (\ref{63}), whereas the fourth line is attained using (\ref{61}) and the understanding that $\mathbb{E}[w(\mathbf{X})\tilde{w}_1(\mathbf{X})]=\tilde{\sigma}_1^2$. Subtracting each of (\ref{sr1}) and (\ref{sr2}) from (\ref{sr3}) yields
\begin{equation}
\bar{e}_{1,l}-\tilde{e}_{1}  =  - \left(1-\alpha_{1,l}\right)^2 \left(\hat{\sigma}_1^2-\tilde{\sigma}_1^2 \right) \le 0,
\nonumber
\end{equation}
\begin{equation}
\bar{e}_{1,l}-\hat{e}_{1}  =  - \alpha_{1,l}^2 \left(\hat{\sigma}_1^2-\tilde{\sigma}_1^2 \right) \le 0,
\nonumber
\end{equation}
where the inequalities follow from Lemma \ref{lem:1}, completing the proof.
\end{proof}

Although the result of Theorem \ref{thm:3} is expected, presenting the theorem and a formal proof is appropriate, given that such result has yet to appear in the literature.  It is less simple to follow suit for the univariate nonlinear hybrid approximation or for a general $S$-variate hybrid approximation.

\subsection{Example 4}

Consider the function
\begin{equation}
y=\left[\frac{2}{N}\sum_{i=1}^{N}X_{i}\right]^{m}\label{72}
\end{equation}
of $N$ independent, identical, and uniformly distributed random variables $X_i$, $i=1,\cdots,N$ over {[}0,1{]}, where $N=10$ and $m\in\mathbb{N}$ is an exponent. The function $y$ in (\ref{72}) is purely additive when $m=1$, but it transitions
from strongly additive to strongly multiplicative as $m$ grows
larger. The objective of this example is to compare univariate ADD, univariate FDD, and univariate HDD (linear and nonlinear) approximations, that is, $\tilde{y}_{1}(\mathbf{X}),$
$\hat{y}_{1}(\mathbf{X})$, $\bar{y}_{1,l}(\mathbf{X};\alpha_{1,l},\beta_{1,l})=\bar{y}_{1,l'}(\mathbf{X};\alpha_{1,l'})$,
and $\bar{y}_{1,n}(\mathbf{X};\alpha_{1,n},\beta_{1,n},\gamma_{1,n})$
in calculating the variance and rare-event probabilities of $y(\mathbf{X})$ for $m=2,3,4,5,6,7,8$.
The second-moment properties of $y(\mathbf{X})$, given $m$, were
calculated exactly. The variances of $\tilde{y}_{1}(\mathbf{X})$
{[}(\ref{54}){]}, $\hat{y}_{1}(\mathbf{X})$ {[}(\ref{55}){]}, and
$\bar{y}_{1,l}(\mathbf{X};\alpha_{1,l},\beta_{1,l})$ {[}(\ref{63}){]},
including all univariate ADD and FDD components functions, were calculated
analytically. The expectations $\mathbb{E}[w(\mathbf{X})\hat{w}(\mathbf{X})]$ and
$\mathbb{E}[w(\mathbf{X})\tilde{w}(\mathbf{X})\hat{w}(\mathbf{X})]$,
involved in determining the optimal
model parameters {[}(\ref{61}),(\ref{65}){]} and the variance of $\bar{y}_{1,n}(\mathbf{X};\alpha_{1,n},\beta_{1,n},\gamma_{1,n})$
{[}(\ref{71}){]}, were estimated by a fully symmetric multidimensional
integration rule with nine generators \cite{genz83}.

\begin{figure}[h]
\begin{centering}
\includegraphics[scale=0.63]{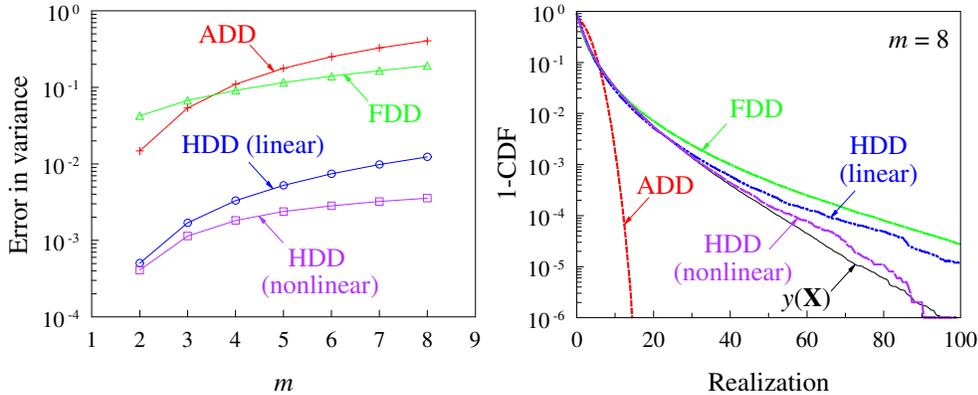}
\par\end{centering}
\caption{Results of univariate ADD, FDD, and HDD approximations of $y$ in Example 4; note: mean-squared errors (left); complementary CDF (right).}
\label{figure2}
\end{figure}

Figure \ref{figure2} (left) plots the respective errors in calculating the variance of
$y$ by four univariate ($S=1$) truncations of ADD, FDD, HDD (linear),
and HDD (nonlinear) against $m$. The definition of the error is the same
as in Example 1 or 2. When $m=2$ or 3, the function is still strongly additive and, therefore, the univariate ADD approximation is better than the univariate FDD approximation.  But the trend reverses when $4 \le m \le 8$, the range of higher values examined. This is because the function switches from dominantly additive ($m \le 3$) to dominantly multiplicative ($m > 3$) as $m$ increases. Nonetheless, for all values of $m$ considered, the univariate HDD approximation, whether linear or nonlinear,
commits lower errors than either univariate ADD or univariate FDD
approximation. The nonlinear version of the univariate HDD approximation
is even more precise than its linear counterpart $-$ a trend that
becomes pronounced when $m$ is larger. However, the improvement of
the nonlinear model comes with a price, as the model requires calculations
of higher-order moments, alluded to in the preceding subsection. Nonetheless, the HDD models proposed provide a means to calculate the second-moment properties more accurately
than ADD and FDD approximations.

Figure \ref{figure2} (right) displays the complementary cumulative distribution functions (CDFs) of $y$ for $m=8$ and its four univariate approximations, each obtained by 10 million Monte Carlo samples. The chosen scale of the vertical axis is logarithmic to delineate rare-event probabilities that are commonly used for reliability analysis of complex systems. Compared with univariate ADD and FDD approximations, both variants of the univariate HDD approximation developed provide better estimates of the tail probabilistic characteristics of $y$.  The ADD approximation significantly underestimates the tail behavior, whereas the FDD approximation overestimates the complementary CDF by a moderate amount.  The nonlinear HDD approximation is more precise than the linear version, especially when the probabilities are very low.  Indeed, an HDD approximation is desirable for uncertainty quantification of  high-dimensional, complex systems, where only univariate truncations are feasible, but not necessarily producing adequate accuracy by either ADD or FDD approximation alone.

\subsection{Example 5}
The final example is motivated on solving a practical engineering problem, which entails eigenspectrum analysis of a piezoelectric transducer commonly used for converting electrical pulses to mechanical vibrations, and vice versa. Figure \ref{figure3} (left) shows a 25-mm-diameter cylinder made of a piezoelectric ceramic PZT4 (lead zirconate titanate) with brass end caps. The thicknesses of the transducer and end caps are 1.5 mm and 3 mm, respectively. The cylinder, 25 mm long, was electroded on both the inner and outer surfaces. The random variables include: (1) ten non-zero constants defining elasticity, piezoelectric stress coefficients, and dielectric properties of PZT4; (2) elastic modulus and Poisson's ratio of brass; and (3) mass densities of brass and PZT4 \cite{rahman09b}. The statistical properties of all 14 random variables are listed in Table \ref{table5}. The random variables are independent and follow lognormal distributions. Due to axisymmetry, a twenty-noded finite-element discrete model of a slice of the transducer, shown in Table \ref{figure3} (right), was created.  The objective is to evaluate various univariate approximations in calculating the second-moment properties of the natural frequencies of the transducer.  For this problem, the ANOVA component functions of each natural frequency response were approximated by third-order polynomial (Hermite) expansions in terms of orthogonal polynomials \cite{rahman08}, where the expansion coefficients were estimated by dimension-reduction integration \cite{xu04}.

\begin{figure}[h]
\begin{centering}
\includegraphics[scale=0.75]{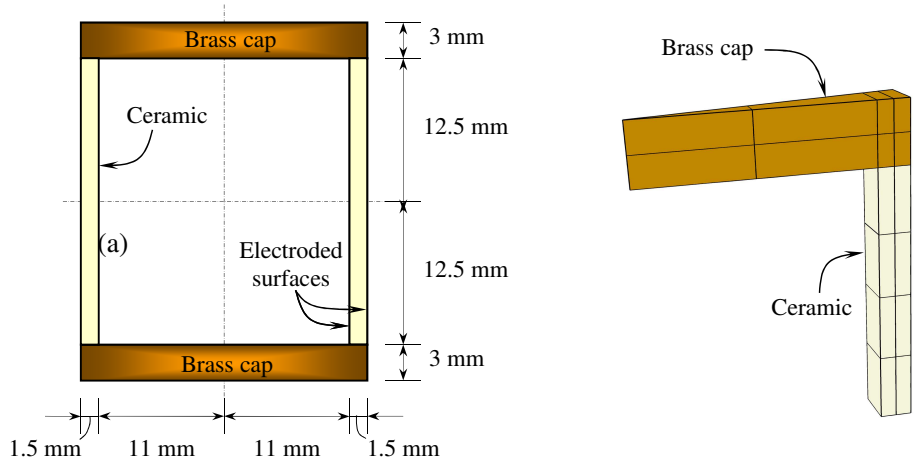}
\par\end{centering}
\caption{A piezoelectric transducer; note: geometry (left); finite-element discrete model (right).}
\label{figure3}
\end{figure}

\begin{table}
\caption{Statistical properties of the random input for the piezoelectric transducer}
\begin{centering}
\begin{tabular}{ccccc}
\hline
Random variable  & Property$^{(\mathrm{a})}$ & Mean & $\begin{array}{c}
\mathrm{\mathrm{Coefficient\; of}}\\
\mathrm{variation}
\end{array}$\tabularnewline
\hline
$X_{1},$ GPa & $D_{1111}$ & 115.4 & 0.15\tabularnewline
$X_{2},$ GPa & $D_{1122},D_{1133}$ & 74.28 & 0.15\tabularnewline
$X_{3},$ GPa & $D_{2222},D_{3333}$ & 139 & 0.15\tabularnewline
$X_{4},$ GPa & $D_{2233}$ & 77.84 & 0.15\tabularnewline
$X_{5},$ GPa & $D_{1212},D_{2323},D_{1313}$ & 25.64 & 0.15\tabularnewline
$X_{6},$ Coulomb/$\mathrm{m}{}^{2}$\ \  & $e_{111}$ & 15.08 & 0.1\tabularnewline
$X_{7},$ Coulomb/$\mathrm{m}{}^{2}$\ \  & $e_{122},e_{133}$ & -5.207 & 0.1\tabularnewline
$X_{8},$ Coulomb/$\mathrm{m}{}^{2}$\ \  & $e_{212},e_{313}$ & 12.71 & 0.1\tabularnewline
$X_{9},$ nF/m & $D_{11}$ & 5.872 & 0.1\tabularnewline
$X_{10},$ nF/m & $D_{22},D_{33}$ & 6.752 & 0.1\tabularnewline
$X_{11},$ GPa & $E_{b}$ & 104 & 0.15\tabularnewline
$X_{12}$ & $\nu_{b}$ & 0.37 & 0.05\tabularnewline
$X_{13},$ g/$\mathrm{m}^{3}$ & $\rho_{b}$ & 8500 & 0.15\tabularnewline
$X_{14},$ g/$\mathrm{m}^{3}$ & $\rho_{c}$ & 7500 & 0.15\tabularnewline
\hline
\end{tabular}
\par\end{centering}

{\footnotesize ~~~~~~~~~~~~~~~~~~~(a) $D_{ijkl}$
are elastic moduli of ceramic; $e_{ijk}$ are piezoelectric stress
coefficients }{\footnotesize \par}

{\footnotesize ~~~~~~~~~~~~~~~~~~~of ceramic; $D_{ij}$
are dielectric constants of ceramic; $E_{b},\nu_{b},\rho_{b}$ are
elastic }{\footnotesize \par}

{\footnotesize ~~~~~~~~~~~~~~~~~~~modulus, Poisson's
ratio, and mass density of brass; $\rho_{c}$ is mass density of ceramic.}
\label{table5}
\end{table}

Tables \ref{table6} and \ref{table7} present the means and standard deviations, respectively, of the first six natural frequencies, $\Omega_i$, $i=1,\cdots,6$, of the transducer by four dimensional decomposition methods: univariate ADD, univariate FDD, univariate HDD (linear), and bivariate ADD approximations; and crude Monte Carlo simulation.  The second-moment properties by decomposition methods, obtained by four-point Gauss-Hermite quadrature rules for estimating the expansion coefficients, are judged to be converged responses, as their changes due to further increases in polynomial order or number of quadrature nodes are negligibly small.  Therefore, the univariate and bivariate approximations require $14\times (14-1)+1=43$ and $14\times(14-1)(4-1)^2/2+(14\times(4-1)+1=862$ finite-element analyses, respectively \cite{rahman08}.  Due to expensive finite-element analysis, crude Monte Carlo simulation was conducted only up to 50,000 realizations, which should be adequate for providing benchmark solutions of the second-moment characteristics. The agreement between the means by approximate and Monte Carlo methods in Table \ref{table6} is excellent.  A comparison of standard deviations in Table \ref{table7} reveals superiority of the bivariate ADD approximation over univariate ADD and FDD approximations, as expected, but at a computational cost markedly higher than the univariate approximations.  More importantly, the univariate HDD approximation proposed is better than either univariate ADD or univariate FDD approximation and produces standard deviations very close to those by the bivariate ADD and crude Monte Carlo methods without carrying the computational burden of the latter methods.

\begin{table}
\caption{Means of natural frequencies of the piezoelectric transducer}
\begin{centering}
\begin{tabular}{cccccc}
\hline
 & \multicolumn{4}{c}{Approximation} & \tabularnewline
\cline{2-5}
Frequency & $\begin{array}{c}
\mathrm{\mathrm{Univariate}}\\
\mathrm{ADD}
\end{array}$ & $\begin{array}{c}
\mathrm{\mathrm{Univariate}}\\
\mathrm{FDD}
\end{array}$ & $\begin{array}{c}
\mathrm{\mathrm{Univariate}}\\
\mathrm{HDD}
\end{array}$ & $\begin{array}{c}
\mathrm{\mathrm{Bivariate}}\\
\mathrm{ADD}
\end{array}$ & $\begin{array}{c}
\mathrm{\mathrm{Crude}}\\
\mathrm{\mathrm{Monte}}\\
\mathrm{\mathrm{Carlo}}
\end{array}$\tabularnewline
\hline
$\Omega_{1}$, kHz & 19.45 & 19.45 & 19.45 & 19.38 & 19.35\tabularnewline
$\Omega_{2}$, kHz & 42.31 & 42.31 & 42.31 & 42.27 & 42.25\tabularnewline
$\Omega_{3}$, kHz & 59.23 & 59.23 & 59.23 & 59.42 & 59.33\tabularnewline
$\Omega_{4}$, kHz & 67.45 & 67.45 & 67.45 & 67.21 & 67.23\tabularnewline
$\Omega_{5}$, kHz & 90.57 & 90.57 & 90.57 & 90.60 & 90.62\tabularnewline
$\Omega_{6}$, kHz & 101.69 & 101.69 & 101.69 & 101.66 & 101.57\tabularnewline
\hline
\end{tabular}
\par\end{centering}
\label{table6}
\end{table}

\begin{table}
\caption{Standard deviations of natural frequencies of the piezoelectric transducer}
\begin{centering}
\begin{tabular}{cccccc}
\hline
 & \multicolumn{4}{c}{Approximation} & \tabularnewline
\cline{2-5}
Frequency & $\begin{array}{c}
\mathrm{\mathrm{Univariate}}\\
\mathrm{ADD}
\end{array}$ & $\begin{array}{c}
\mathrm{\mathrm{Univariate}}\\
\mathrm{FDD}
\end{array}$ & $\begin{array}{c}
\mathrm{\mathrm{Univariate}}\\
\mathrm{HDD}
\end{array}$ & $\begin{array}{c}
\mathrm{\mathrm{Bivariate}}\\
\mathrm{ADD}
\end{array}$ & $\begin{array}{c}
\mathrm{\mathrm{Crude}}\\
\mathrm{\mathrm{Monte}}\\
\mathrm{\mathrm{Carlo}}
\end{array}$\tabularnewline
\hline
$\Omega_{1}$, kHz & 2.30 & 2.30 & 2.68 & 2.54 & 2.66\tabularnewline
$\Omega_{2}$, kHz & 7.03 & 7.06 & 7.08 & 7.09 & 7.11\tabularnewline
$\Omega_{3}$, kHz & 6.65 & 6.66 & 6.74 & 6.82 & 6.83\tabularnewline
$\Omega_{4}$, kHz & 6.94 & 6.95 & 7.00 & 7.09 & 7.00\tabularnewline
$\Omega_{5}$, kHz & 7.37 & 7.38 & 7.48 & 7.58 & 7.51\tabularnewline
$\Omega_{6}$, kHz & 9.42 & 9.43 & 9.35 & 9.29 & 9.29\tabularnewline
\hline
\end{tabular}
\par\end{centering}
\label{table7}
\end{table}

\section{Conclusions and Outlook}

It is time to take stock and recap what has been accomplished so far and what remains to be done.

Two dimensional decompositions, namely, FDD and HDD, of a
multivariate function, representing finite products and sum-product
mixtures of lower-dimensional component functions, were developed. A
theorem, proven herein, reveals the relationship between all component
functions of FDD and ADD, so far available only for univariate and
bivariate component functions. Three function classes, comprising
purely additive functions, purely multiplicative functions, and their
mixtures, were examined to illustrate when and how an FDD approximation
is more precise or relevant than an ADD approximation and vice versa. However, when a
function is not endowed with a specific dimensional hierarchy, an HDD approximation, optimally blending ADD and FDD approximations, is more appropriate than either ADD or FDD approximation.  Furthermore, the FDD and HDD lead to alternative definitions of
effective dimension, reported in the literature associated with ADD
only. New closed-form or analytical expressions were derived for calculating the variances stemming from univariate truncations of all three decompositions. The subsequent mean-squared error analysis pertaining to univariate ADD, FDD, and HDD approximations finds appropriate conditions when one approximation is better than the other. Numerical results from four simple yet insightful examples and a practical engineering problem indicate that an HDD approximation, when called for, commits lower
errors than does ADD or FDD approximation. Therefore, HDD, whether formed
linearly or nonlinearly, is ideally suited to a general function approximation
that may otherwise require higher-variate ADD or FDD truncations for
rendering acceptable accuracy in stochastic solutions.

Future work involves developing an adaptive strategy for FDD and HDD approximations to solve industrial-scale, stochastic problems encountered in engineering and applied sciences.

\section*{Acknowledgments}

The author would like to acknowledge financial support from the U.S.
National Science Foundation under Grant Nos. CMMI-0969044 and CMMI-1130147.

\end{document}